\documentclass[a4paper, 10pt, english, reqno, dvipsnames, table]{amsart}

\usepackage{amsmath}
\usepackage{amsthm}
\usepackage{amssymb}
\usepackage{mathtools}  
\usepackage{mathrsfs}   
\usepackage{amsfonts}

\usepackage{multirow}

\usepackage{amsbsy}

\usepackage{lmodern}

\usepackage{tikz}
\usepackage{tikz-cd}
\usepackage{quiver}

\usepackage{geometry}
\usepackage[T1]{fontenc}  
\usepackage{lmodern}

\theoremstyle{plain}
\newtheorem{theorem}{Theorem}[section]
\newtheorem{proposition}[theorem]{Proposition}
\newtheorem{lemma}[theorem]{Lemma}
\newtheorem{corollary}[theorem]{Corollary}

\newtheorem{definition-theorem}[theorem]{Definition-Theorem}
\newtheorem{definition-proposition}[theorem]{Definition-Proposition}
\newtheorem{maintheorem}{Theorem}

\theoremstyle{definition}

\newtheorem{example}[theorem]{Example}

\theoremstyle{remark}
\newtheorem{remark}[theorem]{Remark}

\usepackage{paralist}   
\usepackage[figurewithin=none]{caption}
\usepackage{graphicx}
\usepackage{float}
\usepackage{xstring}

\usepackage{xcolor}

\usepackage{myarrows}
\usepackage{mycom}

\makeatletter
\def\l@subsection{\@tocline{2}{0pt}{2pc}{5pc}{}}
\makeatother

\usepackage[style=alphabetic,sorting=nyt, backend=biber, backref=true]{biblatex}
\usepackage{url}

\usepackage[unicode, colorlinks, linktocpage=true]{hyperref}
\hypersetup{
	linkcolor = blue,
	citecolor = Red,
	urlcolor = PineGreen,
	bookmarksnumbered = true
}

\usepackage{orcidlink}

\title[Rank of CMAVs over anticyclotomic towers]{On the Mordell-Weil rank of certain CM abelian varieties over anticyclotomic towers}
\author{Haidong Li \orcidlink{0009-0008-6381-2600}}
\address[Haidong Li]{Jiangsu Police Institute. No.48 Shifo Sangong, Pukou District, Nanjing, 210031, China.}
\email{lihaidong@jspi.cn}

\author{Ruichen Xu \orcidlink{0009-0003-4555-2116}}
\address[Ruichen Xu]{Academy of Mathematics and Systems Science, Chinese Academy of Sciences. No. 55 Zhongguancun East Road, Haidian District, Beijing, 100190, China.}
\email{xuruichen21@mails.ucas.ac.cn}

\date{July 29, 2026}

\begin{document}
\begin{abstract}
Let $K/\QQ$ be an imaginary quadratic extension, and let $p$ be an odd prime. In this paper, we investigate the growth of Mordell-Weil ranks of certain CM abelian varieties associated with Hecke characters over $K$ of infinite type $(1, 0)$ along the anticyclotomic $\ZZ_p$-tower of $K$. Our results cover \emph{all} decomposition types of $p$ in $K$. The analytic aspect of our proof is based on our computations of the local and global root numbers of Hecke characters, together with a recent generalization by Haijun Jia of David Rohrlich's result concerning the relation between the vanishing orders of Hecke $L$-functions and their root numbers. The arithmetic conclusions then follow from the Gross-Zagier formula and the Kolyvagin machinery.
\end{abstract}  

\maketitle

\tableofcontents

\section{Introduction}
Let $A$ be an abelian variety over $\QQ$, and let $L/\QQ$ be a number field. The famous (generalized) Mordell-Weil theorem states that $A(L)$ is a finitely generated abelian group. The $\ZZ$-rank of $A(L)$, known as the \emph{Mordell-Weil rank} of $A$ over $L$, is a key invariant of the abelian variety $A$. In his 1983 ICM report \cite{MR804682}, Barry Mazur initiated the study of the growth of the Mordell-Weil rank as $L$ varies along certain towers of number fields over $\QQ$, focusing on the case where $A$ is an elliptic curve.

Let $K/\QQ$ be an imaginary quadratic extension, and let $\varphi: \AA_{K}^{\times} \rightarrow \CC^{\times}$ be a unitary \emph{anticyclotomic} Hecke character of infinite type $(1,0)$. The term "anticyclotomic" means $\varphi \circ \sfc = \bar{\varphi}$, where $\sfc$ denotes the complex conjugation of $K$. By the work of Hecke and Shimura, one can attach a modular form $\theta_{\varphi}$ to $\varphi$. Furthermore, by the construction of Eichler and Shimura, one can canonically attach an abelian variety $A_{\varphi}$ over $\QQ$ to $\theta_{\varphi}$. This abelian variety is of $\GL_2$-type and has complex multiplication by the imaginary quadratic field $K$, which means $K$ is contained in the CM field of $A_{\varphi}$.

Let $p$ be an odd prime number, and let $K_\infty^{\ac}/K$ denote the anticyclotomic $\ZZ_p$-extension of $K$, with its $n$-th layer denoted by $K_n^{\ac}$ such that $\mathrm{Gal}(K_n^{\ac}/K) = \ZZ/p^n\ZZ$. We set $K_0^{\ac} = K$. In this paper, we study the growth of the Mordell-Weil rank of $A_{\varphi}$ along this anticyclotomic tower. The growth behavior depends significantly on the reduction type of $p$ in $K$ and the root number $W(\varphi)$ of the Hecke character $\varphi$.

We prove in Section \ref{sec:mainresult} that
\[
\rank_{\ZZ} A_{\varphi}(K_{n}^{\ac}) \equiv \rank_{\ZZ} A_{\varphi}(K_{n-1}^{\ac}) \mod{\phi(p^n) = p^{n-1}(p-1)}.
\]
Thus, we can write
\[
\rank_{\ZZ} A_{\varphi}(K_{n}^{\ac}) - \rank_{\ZZ} A_{\varphi}(K_{n-1}^{\ac}) = \epsilon_n \phi(p^n),
\]
where $\epsilon_n$ are nonnegative integers. The following main theorem describes $\epsilon_n$.

\begin{maintheorem} \label{thm:mainA}
Let $A_{\varphi}$ be the abelian variety as described above, and let $d = \dim A_{\varphi}$. Let $p > 2$ be a prime number. Denote $\tilW(\varphi) := (1 - W(\varphi))/2$. Then for $n \gg 0$:
\begin{enumerate}[\rm (i)]
    \item If $p$ splits in $K$, then $\epsilon_n = 0$ if $W(\varphi) = 1$, and $\epsilon_n = 2d$ if $W(\varphi) = -1$.
    \item If $p$ remains inert in $K$, let $j$ be the nonnegative integer so that $H_K \cap K_{\infty}^{\ac} = K_{j}^{\ac}$, where $H_K$ is the Hilbert class field of $K$, and let $f(\varphi_p)$ be the exponent of the conductor of the $p$-component of $\varphi$. Then
    \begin{enumerate}[\rm (a)]
        \item if $j+f(\varphi_p)$ is even, then $\epsilon_n = \begin{cases}
        2 d, &\quad\text{if } n \equiv \tilW(\varphi) \bmod 2, \\
        0, &\quad\text{if } n \not\equiv \tilW(\varphi) \bmod 2,
        \end{cases}$
        \item if $j+f(\varphi_p)$ is odd, then $\epsilon_n = \begin{cases}
        0, &\quad\text{if } n \equiv \tilW(\varphi) \bmod 2, \\
        2d, &\quad\text{if } n \not\equiv \tilW(\varphi) \bmod 2.
        \end{cases}$
    \end{enumerate}
    \item If $p$ is ramified in $K$, then $\epsilon_n = d$.
\end{enumerate}
\end{maintheorem}
We shall provide examples in which $j > 0$ indeed occurs, equivalently, $H_{K} \cap K_{\infty}^{\ac} \supsetneq K$, in Example \ref{eg:nontrivialj}.

\subsection{Previous works} \label{sec:previous}
Previous studies on this topic mainly focus on the case of elliptic curves. Let $E$ be an elliptic curve over $\QQ$ with complex multiplication by $K$. By the work of Deuring, there exists a Hecke character $\varphi$ over $K$ attached to $E$ such that $L(E,s)=L(\varphi, s)$. The character $\varphi$ is anticyclotomic of infinite type $(1,0)$, and $A_{\varphi} = E$. By the theory of complex multiplication, $K$ must have class number one \footnote{It is well-known that there are only nine imaginary quadratic fields $\QQ(\sqrt{-d})$ of class number one: $d \in \{1,2,3,7,11,19,43,67,163\}$.}. Moreover, the prime $p$ splits (resp. remains inert, ramified) in $K$ if and only if $E$ has good ordinary (resp. good supersingular, additive) reduction at $p$. 

For CM elliptic curves, Theorem \ref{thm:mainA} has been known for a long time in the case where $p$ is unramified in $K$, i.e., $E$ has good reduction at $p$. This phenomenon was first observed by R. Greenberg in \cite{MR700770}. However, the ramified case remains challenging to address with the method in \textit{loc.cit}. See Remark \ref{rmk:greenberg} for a detailed explanation, and Remark \ref{rem:matches} for an examination that Theorem \ref{thm:mainA} matches the previous results in this case.

\subsection{Method of proof}
Our proof of Theorem \ref{thm:mainA} relies on Haijun Jia's generalization \cite{jia2024lfunctionsheckecharactersanticyclotomic} of David Rohrlich's result \cite{MR735332} on certain Hecke $L$-functions, which we summarize below.

We consider the family of \emph{anticyclotomic twists} of $\varphi$, denoted $\frX^{\ac}_{\varphi}$, consisting of the primitive Hecke characters $\chi = \varphi \rho$ where $\rho: \mathrm{Gal}(K_{\infty}^{\ac}/K) \to \CC^{\times}$ runs over continuous characters, regarded as Hecke characters via the Artin reciprocity map. For any $\chi \in \frX^{\ac}_{\varphi}$, the associated Hecke $L$-function $L(\chi, s)$ has a root number $W(\chi) \in \{\pm 1\}$ (see Section \ref{sec:heckechar} for details).

Recently, Haijun Jia proved the following theorem, which generalizes Rohrlich's main theorem in \cite{MR735332} from the case where $K$ has class number one to general $K$.

\begin{theorem}[Rohrlich, Jia] \label{thm:RJ}
For all but finitely many $\chi \in \frX^{\ac}_{\varphi}$, 
\[
\ord_{s=1} L(\chi, s) = 
\begin{cases}
0, & W(\chi) = 1, \\
1, & W(\chi) = -1.
\end{cases}
\]
\end{theorem}

From this theorem, the proof proceeds as follows:
\begin{itemize}
    \item By Theorem \ref{thm:RJ}, for sufficiently large $n > n_{\varphi}$ and any $\rho$ factoring through $\mathrm{Gal}(K_{n}^{\ac}/K)$, the value $\ord_{s=1} L(\chi,s)$ belongs to $\{0, 1\}$ for $\chi = \varphi \rho \in \frX_{\varphi}^{\ac}$. This result enables us to apply the Gross-Zagier formula and Kolyvagin theory (Theorem \ref{lemma:GZK}) to describe the Mordell-Weil rank of $A_{\varphi}(K_n^{\ac})$. However, this approach requires a more detailed analysis of the abelian variety $A_{\varphi}$, which may be more subtle than it initially appears.
    \item Furthermore, the root number $W(\chi)$ determines $\ord_{s=1} L(\chi,s)$ for such $\rho$. We compute these root numbers in terms of $W(\varphi)$ and $\rho$ by first computing local root numbers. These are carried out in Section \ref{sec:rootnumbercompute} with preparations in Section \ref{sec:heckechar}. The final result on the computation of root numbers is Theorem \ref{thm:rootnumber}.
\end{itemize}

The main innovation of our work lies in these computations of root numbers, particularly for the ramified case.

\subsection{Application I: Finiteness of $A_{\varphi}(K_{\infty}^{\ac})$} \label{sec:introfurther}
Let $F$ be a number field. It is known that along any $\ZZ_p$-extension $F_{\infty}/F$, the $p$-primary part of the class groups $\Cl(F_n)[p^{\infty}]$ satisfies the following growth property:
\[
\# \Cl(F_n)[p^{\infty}] = p^{\lambda n + \mu p^n + \nu}, \quad \text{for } n \gg 0,
\]
where $\lambda, \mu$, and $\nu$ are non-negative integers. For $F = \QQ$ and the cyclotomic $\ZZ_p$-extension $\QQ_{\infty}^{\cyc}/\QQ$, it is known that $\mu = 0$. This led to expectations of $\mu = 0$ for a wider class of $\ZZ_p$-extensions.

Historically, the anticyclotomic extension $K_{\infty}^{\ac}/K$ for CM extensions $K/F$ was first studied in \cite[Section 2]{MR357371}, where $K = \QQ(\zeta_{p^n})$ and $F = \QQ(\zeta_{p^n})^{+}$, the maximal totally real subfield of $K$, were considered instead of the imaginary quadratic extension $K/\QQ$ as in this article. In \textit{loc.cit.}, Iwasawa showed that for this $\ZZ_p$-extension, the $\mu$-invariant is \emph{nonzero}. Since then, anticyclotomic $\ZZ_p$-extensions have played an important role in providing counterexamples to phenomena observed in cyclotomic $\ZZ_p$-extensions. This pattern extends to the arithmetic of elliptic curves and abelian varieties.

Consider the case of elliptic curves $E$ as in Section \ref{sec:previous} that have good reduction at $p$. It is shown\footnote{In \cite{MR735333}, Rohrlich proved that the rank $\rank_{\ZZ} E(\QQ_{n}^{\ac})$ is bounded as $n$ tends to infinity. In the appendix of \cite{MR659153} by Ribet, it is shown that the torsion part of $E(\QQ_{\infty}^{\cyc})$ is finite. Combining these results, we conclude that $E(\QQ_{\infty}^{\cyc})$ is a finitely generated abelian group.} in \cite{MR735333} and \cite{MR659153} that $E(\QQ_{\infty}^{\cyc})$ is always a finitely generated abelian group. One then wonders whether $E(F_{\infty})$ is always finitely generated for any $\ZZ_p$-extension $F_{\infty}/F$. The anticyclotomic $\ZZ_p$-extension $K_{\infty}^{\ac}/K$ again provides a counterexample to this expectation, as first observed in \cite{MR700770} (see also \cite[Theorem 1.7, Theorem 1.8]{MR1860044}, although no rigorous proof is found in \textit{loc.cit.}). 

In this article, by using Theorem \ref{thm:mainA}, we provide a sufficient and necessary condition for $A_{\varphi}(K_{\infty}^{\ac})$ to be a finitely generated abelian group. This is made possible by our considerations of the case where $p$ is ramified in $K$, together with a finiteness result on the torsion subgroup of $A_{\varphi}(K_{\infty}^{\ac})$, for which we provide a proof. Details are introduced in Section \ref{sec:finitetor} with the main result stated as Corollary \ref{coro:finitelygenerated}.

\subsection{Application II: Distribution of vanishing orders among anticyclotomic twists} \label{sec:introdistribution}
Let $\varphi$ be an anticyclotomic Hecke character over $K$ of infinite type $(1,0)$ as before. When we deform it in the anticyclotomic $\ZZ_p$-extension $K_{\infty}^{\ac}$ by twisting Galois characters $\rho: \Gal(K_{\infty}^{\ac}/K) \rightarrow \CC^{\times}$, a natural question arises: 
\begin{quote}
    How is the (parity of the) vanishing order of $L(\varphi \rho,s)$ at $s=1$ distributed? 
\end{quote} 

Theorem \ref{thm:RJ} states that for Galois characters $\rho: \Gal(K_{\infty}^{\ac}/K) \rightarrow \CC^{\times}$ of sufficiently large level \( n \), the order of the \( L \)-function \( L(\varphi \rho, s) \) at \( s=1 \) is either \( 0 \) or \( 1 \), completely determined by the root number of \( \varphi \rho \). This result enables us to apply our computation on root numbers, especially Theorem \ref{thm:rootnumber}, to address the aforementioned problem. In this regard, we derive the following result, stated informally here but presented and proved rigorously in Section \ref{sec:distribution}:

\begin{itemize}
    \item When $p$ splits in $K$, either 100\% of the anticyclotomic twists $\varphi \rho$ have even vanishing order at $s=1$, or 0\% of them do, depending on whether $W(\varphi) = 1$ or not.
    \item When $p$ remains inert in $K$, the probabilities of $\varphi \rho$ having even or odd vanishing order at $s=1$ fluctuate in the interval $[1/(p+1), p/(p+1)]$. This fluctuation is characterized by the parity of the level of $\rho$, as well as the root number $W(\varphi)$, the conductor of $\varphi_p$, and the class group of $K$.
    \item When $p$ ramifies in $K$, then 50\% of the anticyclotomic twists $\varphi \rho$ have even vanishing order at $s=1$, and 50\% have odd vanishing order at $s=1$. This phenomenon is independent of the root number $W(\varphi)$.
\end{itemize}

As we can see from the result, the case where $p$ remains inert in $K$ is the most subtle and surprising.

\subsection{Notations}
Let $K$ be any number field, i.e. a finite extension of the field of rational numbers $\QQ$, we denote by
\begin{itemize}
    \item $\scV_{K}$ the set of places of $K$,
    \item $\Gal_K$ the absolute Galois group of $K$,
    \item $\calO_{K}$ the ring of integers of $K$,
    \item $K_v$ the completion of $K$ with respect to a place $v$ of $K$, with the ring of integers $\calO_{K_v}$, a uniformizer $\varpi_{K_v}$ and residue field $k_{K_v}$ of cardinality $q_{K_v}$. If it is clear from contexts, we simply denote these objects by $\calO_{v}, \varpi_{v}$, $k_v$ and $q_v$ respectively.
    \item Suppose $v$ is a finite place of $K$, we denote $\frp_{v}$ be the prime ideal of $\calO_K$ associated to the place $v$.
\end{itemize}
Let $K$ be a local field, i.e. a finite extension of the field of $\ell$-adic rational numbers $\QQ_{\ell}$. We denote by
\begin{itemize}
    \item $\calO_K$ the ring of integers of $K$, and $U_{K}$ be the group of invertible elements of $\calO_{K}$,
    \item $\varpi_K$ a uniformizer of $K$,
    \item $U_{K}^{(m)} := 1 + \varpi_{K}^{m} \calO_{K}$ for integers $m \geq 1$,
    \item $k_{K}$ the residue field of the local ring $\calO_{K}$, with cardinality $q_{K}$,
    \item $\dif_K x$ any Haar measure on $K$, which shall be specified according to our need in the text.
\end{itemize}

One will not be confused by the different use of $K$ since the context is clear.

If it is clear from contexts, we simply denote these objects by $\calO, U, \varpi, k, q$ and $\dif x$ respectively. Other notations will be introduced in the text. Conventions on Hecke characters are introduced in Section \ref{sec:heckechar}. Here we record that $\frX_{?}^{?}, \frY_{?}^{?}$ and $\frZ_{?}$, as sets of certain Hecke characters, are defined in Section \ref{sec:galoischar}, Proposition \ref{lemma:isogeny} and \eqref{eq:defYnpm} in the text.

\subsection*{Acknowledgment} We are deeply grateful to our doctoral advisor, Professor Xin Wan, for his insightful guidance throughout this project. 

This article is a natural continuation of the work initiated by our close friend Haijun Jia in \cite{jia2024lfunctionsheckecharactersanticyclotomic}. We are deeply grateful to him for his insightful contributions and for inspiring
the development of this work. We were deeply saddened to learn of his untimely passing, and we dedicate this article to his memory.

We also extend our thanks to Ashay Burungale, Olivier Fouquet, Jeffrey Hatley, Shinichi Kobayashi, Hang Yin, and Luochen Zhao for their valuable discussions. Prof. Xin Wan introduced us to the paper \cite{MR1133776}, which he discussed with David Rohrlich, while Dr. Luochen Zhao brought \cite{MR4742720} to our attention following his discussions with Kazuto Ota.

Haidong Li completed part of this work during his doctoral studies at Academy of Mathematics and Systems Science, Chinese Academy of Sciences, whose supportive environment he gratefully acknowledges.  Ruichen Xu completed part of this work during his stay at the Laboratoire de Mathématiques de Besançon at the Université Marie \& Louis Pasteur, hosted by Professor Olivier Fouquet. He is sincerely grateful for their generous support and hospitality. This visit is funded by the “2024 International Research Collaboration Program for Postgraduate Students” of the Academy of Mathematics and Systems Science, Chinese Academy of Sciences.

The authors would like to thank the anonymous referee for a careful reading of the manuscript and for many valuable comments and suggestions, which have significantly improved the clarity and quality of the paper.

\section{Preliminaries}
\label{sec:heckechar}

In this section, we review some basic notions on Hecke characters and some results on the class field theory of the anticyclotomic $\ZZ_p$-extension of imaginary quadratic fields.

\subsection{Hecke characters}

In this section, we recall some basic notions on Hecke characters.

\subsubsection{Hecke characters and their $L$-functions}
Let $K$ be any number field, i.e., a finite extension of $\QQ$. In this article, a Hecke character over $K$ is a continuous group homomorphism $\chi : \AA_{K}^{\times} \rightarrow \CC^{\times}$ that is trivial on $K^{\times}$. For each place $v$ of $K$, we denote by $\chi_v : K_v^{\times} \rightarrow \CC^{\times}$ the $v$-component of $\chi$. We consider the decreasing filtration:
\[
\calO_{K_v}^{\times} \supset 1+\varpi_{K_v} \calO_{K_v} \supset 1+\varpi_{K_v}^{2} \calO_{K_v} \supset \cdots \supset 1+\varpi_{K_v}^{e} \calO_{K_v} \supset \cdots.
\]
We say $\chi_v$ is \emph{unramified} if $\chi_v(\calO_{K_v}^{\times}) = 1$; otherwise, it is called \emph{ramified}. When $\chi_v$ is ramified, the \emph{exponent of conductor} of $\chi_v$ is the smallest integer $f(\chi_v)$ such that $\chi_v(1+\varpi_{K_v}^{f(\chi_v)} \calO_{K_v}) = 1$. The \emph{conductor} of $\chi_v$ is then defined as $\frf(\chi_v) := \frp_v^{f(\chi_v)}$, where $\frp_v$ is the prime ideal in $K$. By convention, we set $f(\chi_v) = 0$ if $\chi_v$ is unramified.

Let $K$ be an imaginary quadratic field. We say a Hecke character $\chi$ is of infinite type $(a,b) \in \ZZ^{2}$ if 
\[
\chi_{\infty}(z) = z^{-a} \barz^{-b}.
\]
In his doctoral thesis \cite{MR2612222}, Tate developed a theory of Hecke $L$-functions $L(s,\chi)$ using harmonic analysis over local and global fields. We briefly recall some basic facts about these $L$-functions:
\begin{itemize}
    \item Let $\Lambda(s,\chi)$ be the complete $L$-function. If $k = a+b+1$, then there exists a complex number $W(\chi)$ with $\abs{W(\chi)} = 1$ such that
    \[
    \Lambda(s,\chi) = W(\chi)\Lambda(k-s, \barchi).
    \]
    The constant $W(\chi)$ is called the \emph{root number} of $\chi$. From the functional equation, the central point is $s = k/2$.
    \item The root number depends only on the unitary part of $\chi$, i.e., $W(\chi) = W(\chi_0)$, as introduced in \cite[Section 5.1 of Lecture 2]{MR2882696}. Thus, assuming $\varphi$ is a unitary character does not lose generality.
    \item Since the complex conjugation $\sfc$ permutes the ideals of $\calO_{K}$, $L(\chi, s) = L(\chi \circ \sfc, s)$. Consequently, if $\chi$ is \emph{anticyclotomic}, i.e., $\chi \circ \sfc = \barchi$, the functional equation becomes
    \[
    \Lambda(s,\chi) = W(\chi)\Lambda(k-s, \chi),
    \]
    with $W(\chi) \in \{\pm 1\}$.
\end{itemize}

The following proposition provides equivalent characterizations of anticyclotomic characters (see \cite[Proposition 1]{MR658544}).
\begin{proposition} \label{prop:anticychar}
Let $\chi$ be a unitary Hecke character over $K$ of infinite type $(1,0)$. Then the following statements are equivalent:
\begin{enumerate}[\rm (a)]
    \item $\chi$ is anticyclotomic;
    \item $\chi \vert_{\AA_{\QQ}^{\infty}} = \kappa$, where $\kappa$ is the quadratic character attached to the extension $K/\QQ$ by class field theory.
\end{enumerate}
Moreover, if $\chi$ is anticyclotomic, it is ramified at all primes of $K$ that are ramified over $\QQ$.
\end{proposition}

\begin{remark}
Let $A$ be an elliptic curve over $\QQ$ with complex multiplication by $K$, and let $\varphi$ be the Grössencharacter attached to $A$. By construction, $\varphi$ is anticyclotomic of infinite type $(1,0)$.
\end{remark}

\subsubsection{Galois characters} \label{sec:galoischar}
Let $L/K$ be any abelian extension of number fields, and let $\rho: \Gal(L/K) \rightarrow \CC^{\times}$ be a character of the Galois group $\Gal(L/K)$. By global class field theory, there is a canonical isomorphism called the Artin reciprocity map
\[
\Art_{L/K}:\AA_{K}^{\times}/ N_{L}  \xrightarrow{\sim} \Gal(L/K),
\]
where $N_{L}$ is an open subgroup of $\AA_{K}^{\times}$ containing $K^\times$. Via $\Art_{L/K}$, we can view $\rho$ as a Hecke character via the composition
\begin{equation} \label{eq:identifygaloishecke}
    \AA_{K}^{\times} \twoheadrightarrow \AA_{K}^{\times}/ N_{L} \xrightarrow{\Art_{L/K}} \Gal(L/K) \xrightarrow{\rho} \CC^{\times}.
\end{equation}
For notational simplicity, we still denote this character by $\rho$.

For each positive integer $n$, we define the following sets of Hecke characters:
\begin{itemize}
    \item Let $\frY$ be the set of all Galois characters $\rho: \Gal(K_{\infty}^{\ac}/K) \rightarrow \CC^{\times}$. They are regarded as Hecke characters via \eqref{eq:identifygaloishecke}. Throughout this article, we will simply call $\rho \in \frY$ as a \emph{Galois character}, if this will not cause any confusion. 
    \item For any integer $n \geq 0$, we define $\frY_{n}$ to be the subset of $\frY$, consisting of Galois characters $\rho: \Gal(K_{\infty}^{\ac}/K) \rightarrow \CC^{\times}$ factoring through $\Gal(K_{n}^{\ac}/K)$. We denote $\frY_{n}^{\dagger} := \frY_{n} \smallsetminus \frY_{n-1}$, and characters $\rho \in \frY_{n}^{\dagger}$ are said to be of \emph{level} $n$.
\end{itemize}
Given an anticyclotomic Hecke character $\varphi$ of infinite type $(1,0)$, we define the sets
\[
\frX_{\varphi}^{\ac} := \{\varphi \rho: \rho \in \frY \}, \quad \frX_{\varphi, n}^{\ac} := \{\varphi \rho: \rho \in \frY_n \}, \, \text{ and } \frX_{\varphi,n}^{\ac, \dagger} := \{\varphi \rho: \rho \in \frY_n^{\dagger} \}.
\]
We also say $\chi \in \frX_{\varphi, n}^{\ac, \dagger}$ is of \emph{level} $n$.

Assuming $\varphi$ is anticyclotomic, let $\rho: \Gal(K_\infty^{\ac}/K) \rightarrow \CC^{\times}$ be a Galois character. Viewing it as a Hecke character of $\calK$, we have $\rho \circ \sfc = \overline{\rho}$. Therefore, every character $\chi \in \frX_{\varphi}^{\ac}$ is anticyclotomic.
   
\subsection{Class field theory of anticyclotomic $\ZZ_p$-extensions}\label{subsection:localclassfieldtheory}

In this subsection, we recall basic facts from the class field theory of anticyclotomic $\ZZ_p$-extensions, which will be used in computing the root numbers.

We consider a more general setting as follows. Let $K/F$ be a CM extension of number fields, i.e. $F$ is a totally real number field and $K$ is a quadratic totally imaginary field extension of $F$. Let $w$ be a finite place of $F$ above $p$, and $v$ a finite place of $K$ above $w$.

Let $K_{\infty}^{\ac}$ denote an anticyclotomic $\ZZ_p$-extension of $K$, meaning that the action of $\Gal(K/F)$ on $\Gal(K_{\infty}^{\ac}/K)$ by conjugation is $-1$. Let $K_{n}^{\ac}$ be the intermediate field such that $\Gal(K_{n}^{\ac}/K) \simeq \ZZ/p^n\ZZ$. Also, let $v_n$ and $v_{\infty}$ denote places of $K_{n}^{\ac}$ and $K_{\infty}^{\ac}$, respectively, such that $v_\infty \mid v_n \mid v$ for all $n \geq 1$. Denote by $F_w$, $K_v$, $K_{n,v_n}$, and $K_{\infty, v_\infty}$ the corresponding completions at these places. Furthermore, let $G_n := \Gal(K_n^{\ac}/K)$ and $D_n := \Gal(K_{n,v_n}/K_v)$ be the Galois group and the decomposition group of $v_n \mid v$, respectively. These notations are summarized in the following diagram:
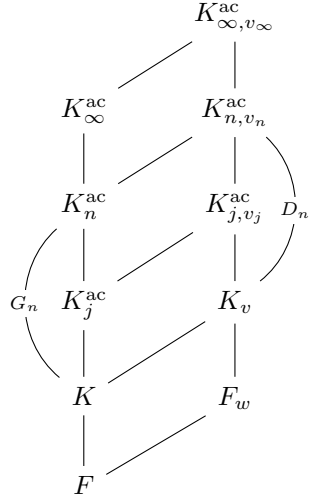
\[
\begin{tikzcd}
     & {K_{\infty,v_\infty}} \arrow[d, no head]                                            \\
    K_\infty^{\ac} \arrow[d, no head] \arrow[ru, no head]                                                                     & {K_{n,v_n}} \arrow[d, no head] \arrow[dd, "D_n" description, no head, bend left=49] \\
    K_n^{\ac} \arrow[ru, no head] \arrow[d, no head] \arrow[dd, "G_n" description, no head, bend right=49] & {K_{j,v_j}} \arrow[d, no head]                                                      \\
    K_j^{\ac} \arrow[ru, no head] \arrow[d, no head]                                                                          & K_v \arrow[d, no head]                                                              \\
    K \arrow[ru, no head] \arrow[d, no head]                                                                            & F_w                                                                                 \\
    F \arrow[ru, no head]                                                                                               &                                                                                    
\end{tikzcd}
\]

The decomposition of a place $v$ of $K$ in $K_{\infty}^{\ac}$ falls into the following two types (see \cite[Proposition 13.2, Lemma 13.3]{washington1997}):

\begin{itemize}
    \item The place $v$ is unramified in $K_{\infty}^{\ac}$. This is the case for all $v$ not dividing $p$. For this type, the local class field theory is quite explicit.
    \item The place $v$ is unramified in $K_j^{\ac}/K$ for some $j \geq 1$, and each place $v_j$ of $K_j^{\ac}$ above $v$ is totally ramified in $K_\infty^{\ac}$. Let $H_K$ be the Hilbert class field of $K$, then $j$ is the nonnegative integer such that $H_K \cap K_\infty^{\ac} = K_j^{\ac}$. In particular, when $p$ does not divide the class number of $K$ (which is automatic when $K$ has class number one), we have $j=0$. This case can only happen when $w$ divides $p$, but not all places of $K$ above $p$ fall in this case. We write
    \[
    \frp_v \calO_{K_{j}^{\ac}} = \prod_{l = 1}^{p^g} \frp_{v_j^{(l)}},
    \]
    that is, $\frp_v$ decomposes into $p^g$ primes in $K_{j}^{\ac}$. Then for $n \geq j$, one checks that $D_{n} = \ZZ/p^{n-g}\ZZ$.
\end{itemize}

We provide a few examples to show that it is indeed possible that 
\[
H_{K} \cap K_{\ac}^{\infty}  \neq K, \quad \text{i.e. } j > 0.
\]
For this purpose, we use the following results in \cite[Corollary on p.~59]{minardi1987} and \cite[Theorem 3]{oh2015anticyclotomic}.

\begin{theorem} \label{thm:minardi}
    Let $K = \QQ(\sqrt{-D})$ be an imaginary quadratic field with $D \not\equiv 3 \pmod{9}$ and $p = 3$, its $p$-Hilbert class field $L_K$ is contained in the anticyclotomic $\ZZ_3$-extension $K_{\infty}^{\ac}/K$ if and only if the class number of the real quadratic field $K^{\sharp} := \QQ(\sqrt{3D})$ is not divisible by $3$.
\end{theorem}

\begin{example} \label{eg:nontrivialj}
We now provide three examples when $H_{K} \cap K_{\infty}^{\ac} \neq K$ when $p = 3$. The first two examples lies in \cite[Table 6.1]{minardi1987} and the third example comes from \cite[Proof of Theorem 2]{oh2015anticyclotomic}.
\begin{itemize}
    \item Let $K = \QQ(\sqrt{-23})$, then $K^{\sharp} = \QQ(\sqrt{69})$ with class number $1$. By Theorem \ref{thm:minardi}, we have $H_{K} \subset K_{\infty}^{\ac}$, and one computes that the class number of $K$ is $3$, so $H_{K} \cap K_{\infty}^{\ac} = K_1^{\ac}$. Additionally, $p$ splits in $K$.
    \item Let $K = \QQ(\sqrt{-231})$, then $K^{\sharp} = \QQ(\sqrt{693}) = \QQ(\sqrt{77})$ with class number $1$. By Theorem \ref{thm:minardi}, we have $H_{K} \subset K_{\infty}^{\ac}$, and one computes that the class number of $K$ is $12$, so $H_{K} \cap K_{\infty}^{\ac} = K_1^{\ac} \subsetneq H_K$. This example shows that $H_{K} \cap K_{\infty}^{\ac}$ may be strictly larger than $K$ but need not exhaust the full Hilbert class field $H_K$. Additionally, $p$ ramifies in $K$.
    \item Let $K = \QQ(\sqrt{-1423})$, then $K^{\sharp} = \QQ(\sqrt{4269})$ with class number $1$. By Theorem \ref{thm:minardi}, we have $H_{K} \subset K_{\infty}^{\ac}$, and one computes that the class number of $K$ is $9$, so $H_{K} \cap K_{\infty}^{\ac} = K_2^{\ac}$. This example shows that it is possible to have $j \geq 2$. Additionally, $p$ remains inert in $K$.
\end{itemize}    
\end{example}

\begin{proposition} \label{prop:localtotram}
With the notations above, if $v \mid w$ is inert or ramified, then $g=j$ and $K_{n,v_n}^{\ac} = K_v$ for $n \leq j$.
\end{proposition}
\begin{proof}
Since $v \mid w$ is inert or ramified, $K_v/F_w$ is a quadratic extension. By local class field theory (see, for example, \cite[Remark 1.8.2]{MR1941965}), $K_v$ admits a unique unramified $\ZZ_p$-extension $L$, obtained by adjoining suitable roots of unity, and $[K_v:\QQ_p]$ independent totally ramified $\ZZ_p$-extensions. The extension $L/F_w$ is abelian as $L$ is the compositum of $K_v$ and the unique unramified $\ZZ_p$-extension of $F_w$. Hence, $L/K_v$ is not anticyclotomic. Consequently, since $K_{\infty, v_{\infty}}^{\ac}/K_v$ is anticyclotomic, it must be totally ramified. Thus, $g=j$ and $K_{n,v_n}^{\ac} = K_v$ for $n \leq j$.
\end{proof}

Let $\rho: \Gal(K_{\infty}^{\ac}/K) \rightarrow \CC^{\times}$ be a Galois character of level $n$, regarded as a Hecke character by global class field theory, as introduced in Section \ref{sec:galoischar}. By the local-global compatibility of the Artin reciprocity map, the $v$-component of $\rho$ can be viewed as a continuous group homomorphism
\begin{equation} \label{eq:identifyDn}
    \rho_v: K^\times_{v}/\Norm((K^{\ac}_{n,v_n})^{\times}) \xrightarrow{\Art_{K_v}} D_n \rightarrow \CC^{\times},
\end{equation}
where $\Art_{K_v}$ is the local Artin reciprocity map, and $\Norm: (K^{\ac}_{n,v_n})^{\times} \rightarrow K_v^{\times}$ is the norm map. 

From this point onward, we assume that $F=\QQ$ and that $K$ is an imaginary quadratic extension of $\QQ$. We write $K=\QQ(\sqrt{-D})$, where $D$ is a positive square-free integer. For each \emph{ramified} local extension $K_v/\QQ_w$, we fix the uniformizer $\varpi_{K_v}:=\sqrt{-D}$. The basic properties of the associated local characters $\rho_v$ are summarized in the following proposition.

\begin{proposition} \label{prop:localrho}
With the notations as above, we have the following.
\begin{enumerate}[\rm (i)]
    \item When $v \nmid p$, the character $\rho_v$ is unramified.
    \item When $v \mid p$ and $v \mid w$ is ramified, then $\rho_v(\varpi_{K_v}) = 1$. Moreover,
    \[
    f(\rho_v) = \begin{cases}
    0, &\quad \text{when } n \leq j, \\
    n-j+1, &\quad \text{when } w \mid p \text{ is inert and } n > j, \\
    2(n-j), &\quad \text{when } w \mid p \text{ is ramified and } n > j.
    \end{cases}
    \]
\end{enumerate}
\end{proposition}

\begin{proof}
    Conclusion (i) follows from the fact that $K_{\infty}^{\ac}/K$ is unramified outside $p$. For (ii), we note that $\rho_v(\varpi_{K_v})^2 = 1$ and $\rho_v(\varpi_{K_v})^{p^n} = 1$, recalling that we have fixed the uniformizer $\varpi_{K_v} = \sqrt{-D}$. Since $p$ is an odd prime, this forces $\rho_v(\varpi_{K_v}) = 1$. Next, we compute the conductor of $\rho_v$. By \eqref{eq:identifyDn}, the main task is to study the structure of 
    \[
    Q_k := 1+\varpi_{K_v} \calO_{K_v}/1+\varpi_{K_v}^k \calO_{K_v}, \quad \text{for } k \geq 1.
    \]
    The extension $K_v/F_w$ is nontrivial of degree $2$, and hence $\sfc$ acts on $Q_k$ nontrivially by $\pm 1$, where $\sfc \in \Gal(K_v/F_w)$ is the nontrivial element, namely the complex conjugation. Therefore, $Q_{k}$ can be decomposed as
    \[
    Q_k \simeq Q_k^{+} \times Q_k^{-},
    \]
    where $Q_{k}^{+}$ (resp. $Q_{k}^{-}$) is the subgroup of $Q_{k}$ where $\sfc$ acts trivially (resp. by $-1$). Throughout the proof, we denote $C_n$ as the cyclic group of order $n$. We discuss the following two cases. 

    \textbf{Case 1: $p$ remains inert in $K$}. Note that in this case $\varpi_{K_v} = p $ and  one checks that every element of $Q_{k}$ is of order at most $p^{k-1}$, and
    \[
    (1+\varpi_{K_v})^{p^{k-1}} = 1 \in Q_{k}, \quad (1+\varpi_{K_v})^{p^{k-2}} \neq 1 \in Q_{k},
    \]
    by the binomial theorem. This shows that $1+\varpi_{K_v}$ generates a cyclic subgroup of order $p^{k-1}$ in $Q_{k}$. One counts that the cardinality of $Q_{k}$ is $p^{2(k-1)}$ since $v \mid w$ is inert.
    
    We start with $k=2$. Since $K_v/F_w$ is inert, one checks that $Q_2^{+}$ and $Q_{2}^{-}$ are both nonzero and share the same cardinality and they can be annihilated by $p$. In fact, one observes that
    \[
    1+pt \in Q_2^+, \quad 1+pt\sqrt{-D} \in Q_2^-, \, \text{ for } t\in \mathbb{Z}_p.
    \]
    Moreover, since $Q_2^+$ and $Q_2^-$ both have $p$ elements, $Q_2^+ \simeq Q_2^- \simeq C_p$.
    
    For $k>2$, since there is a surjective map $Q_k \twoheadrightarrow Q_{k-1}$, which maps the plus (resp. minus) part of $Q_{k}$ into the plus (resp. minus) part of $Q_{k-1}$, one shows by induction that
    \[
    Q_{k} \simeq Q_{{k}}^+ \times Q_{{k}}^-,
    \]
    with $Q_{k}^{+}$ and $Q_{k}^{-}$ both isomorphic to $C_{p^{k-1}}$ as abstract groups.
    
    If $n > j$, then $D_n$ is a group of order $p^{n-j}$, which can be regarded as a quotient group of $1+\varpi \calO_{K_v}$ by local class field theory. Now by the discussion above,
    \[
    D_n \simeq Q_{n-j+1}^- \simeq C_{p^{n-j}}.
    \]
    This implies that $n-j+1$ is the exponent of the conductor of $\rho_v$.

    \textbf{Case 2: $p$ ramifies in $K$}. The analysis is similar to \textbf{Case 1}, but more complicated. Note that in this case, $\varpi_{K_v} = \sqrt{-D}$, one checks similarly that every element of $Q_{2k+1}$ is of order at most $p^k$, and
    \[
    (1+\varpi_{K_v})^{p^{k}} = 1 \in Q_{2k+1}, \quad (1+\varpi_{K_v})^{p^{k-1}} \neq 1 \in Q_{2k+1}.
    \]
    Hence $1+\varpi_{K_v}$ generates a cyclic subgroup of order $p^k$ in $Q_{2k+1}$.
    
    Similarly, we start with $Q_2$ again. However, in this case, the residue extension of $K_{v}/F_{w}$ is trivial, and hence cardinality of $Q_2$ is $p$. Moreover, since elements in $Q_2$ are of the form $1 + x\sqrt{-D}$, it is immediate to see that $Q_2 = Q_{2}^{-} \simeq C_p$.
    
    Next we compute $Q_{3}$. One observes that
    \[
    1 - Dt \in Q_3^+, \quad 1-t\sqrt{-D} \in Q_3^-, \, \text{ for } t\in \ZZ_p,
    \]
    one can deduce that $Q_{3}^{+}$ and $Q_{3}^{-}$ are both isomorphic to $C_p$ as abstract groups. Hence 
    \[
    Q_3 \simeq Q_3^+ \times Q_3^- \simeq C_p \times C_p.
    \]
    Using surjective maps $Q_{k} \twoheadrightarrow Q_{k-1}$, we inductively show that the minus and plus parts increase alternately as $k$ increases. More precisely, for $k\ge 1$,
    \[
    Q_{2k+1} \simeq Q_{2k+1}^+ \times Q_{2k+1}^-
    \]
    with $Q_{2k+1}^+ \simeq C_{p^k}$, $Q_{2k+1}^-\simeq C_{p^k}$, and
    \[
     Q_{2k} \simeq Q_{2k}^{+} \times Q_{2k}^{-}
    \]
    with  $Q_{2k}^+ \simeq C_{p^{k-1}}$, $Q_{2k}^-\simeq C_{p^k}$.

    If $n > j$, then $D_n$ is a group of order $p^{n-j}$, which can be regarded as a quotient group of $1+\varpi \calO_{K_v}$ by local class field theory. Now by the discussion above, 
    \[
    D_{n} \simeq Q_{2(n-j)}^- \simeq C_{p^{n-j}}
    \]
    This implies that $2(n-j)$ is the exponent of the conductor of $\rho_v$.
\end{proof}
 
\begin{remark}
Readers may have noticed that we have not dealt with the case when $w$ splits in $K$. This is because computing local root numbers in this case is straightforward (see Proposition \ref{prop:splitcase}), so there is no need to invoke local class field theory.
\end{remark}

\section{Root number computations} \label{sec:rootnumbercompute}
By Theorem \ref{thm:RJ}, the vanishing order $\ord_{s=1} L(\varphi \rho, s) \in \{0,1\}$ for all but finitely many $\varphi \rho \in \frX_{\varphi}^{\ac}$, and hence it is completely determined by the root number of $\varphi \rho$. In this section, we compute the root number of such characters, which therefore governs the behavior of $L(\varphi \rho, s)$.

\subsection{Background on local root numbers}
\subsubsection{Local root numbers} \label{sec:localrt}
We recall the basic definitions and properties of local root numbers. Let $H$ be a local field of characteristic zero, and let $\chi$ be a \emph{unitary} Hecke character over $H$. Fix a nontrivial unitary additive character $\psi: H \rightarrow \CC^{\times}$. The \emph{local root number} $W(\chi, \psi)$ is defined by
\[
W(\chi, \psi) := \epsilon(\chi, \psi, \dif x)/ \abs{\epsilon(\chi, \psi, \dif x)},
\]
where $\dif x$ is any Haar measure on $H$, and the epsilon factor is defined as in \cite[(3.3.3)]{MR349635}.

Suppose $H$ is nonarchimedean. Let $\calO$ denote its ring of integers, $\varpi$ a uniformizer of $\calO$, and $q$ the order of the residue field $k = \calO/\varpi \calO$. We write:
\begin{itemize}
    \item $m(\psi)$ as the largest integer $\mu$ such that $\psi$ is trivial on $\varpi^{-\mu}\calO$, 
    \item $f(\chi)$ such that $\varpi^{f(\chi)} \calO$ is the conductor of $\chi$. Note that $f(\chi) \geq 0$, and $\chi$ is ramified if and only if $f(\chi) > 0$.
\end{itemize}

In practice, there is a canonical choice of the additive character $\psi$. Define $\psi^{\star}: H \rightarrow \CC^{\times}$ by
\[
\psi^{\star}(t) = \begin{cases}
\exp(-2 \pi i \tr_{H/\RR}(t)), &\text{archimedean } H, \\
\exp (2 \pi i \{\tr_{H/\QQ_{\ell}}(t)\}), &\text{nonarchimedean } H \ (\text{over } \QQ_{\ell}),
\end{cases}
\]
where $\{\cdot\}$ denotes the fractional part.

We summarize the results on the root number of $\chi$ in the following proposition.

\begin{proposition} \label{prop:localrt}
With the notations as above, suppose $\chi$ is of infinite type $(1,0)$.
\begin{enumerate}[\rm (i)]
    \item If $H$ is archimedean, then $W(\chi, \psi^{\star}) = -i$.
    \item If $H$ is nonarchimedean and $\chi$ is unramified, then $W(\chi, \psi) = \chi(\varpi)^{m(\psi)}$.
    \item If $H$ is nonarchimedean, then for any $\beta \in \varpi^{-f(\chi)-m(\psi)} \calO^{\times}$, we have
            \[
            W(\chi, \psi) = \dfrac{q^{f(\chi)/2}}{\Meas(\calO)} \chi(\beta)^{-1} \int_{U} \chi^{-1}(x) \psi(\beta x) \dif x.
            \]
    \end{enumerate}
\end{proposition}

\begin{proof}
Item (i) appears in \cite[Proposition 12]{MR658544}. Item (ii) appears in \cite[Proposition 11]{MR658544}. The integral formula in (iii) appears as \cite[(3.4.3.2), (5.7.2)]{MR349635} (see also \cite[(3)]{MR1133776}).
\end{proof}

\subsubsection{Relative root numbers}
Next, we recall the background on relative local root numbers. Let $F$ be a local field of characteristic zero, and let $K$ be an $F$-algebra of one of the following three types:
\begin{enumerate}[(a)]
    \item $K$ is a ramified quadratic extension of $F$,
    \item $K$ is an unramified quadratic extension of $F$,
    \item $K = F \oplus F$.
\end{enumerate}
We fix nontrivial unitary additive characters $\psi_{K}$ and $\psi_{F}$ of $K$ and $F$, respectively, such that
\[
\psi_{K} = \psi_{F} \circ \tr_{K/F},
\]
where $\tr_{K/F}$ is the trace of the $F$-algebra $K$.

Let $\chi$ be any unitary character of $K^{\times}$ such that $\chi|_{F^{\times}} = \kappa$, where $\kappa$ is the quadratic character of $F^{\times}$ associated with the extension $K/F$. We clarify the situation for type (c). In type (c), the character $\kappa$ is trivial. Writing $\chi = \chi_1 \oplus \chi_2$ on $F^{\times} \oplus F^{\times}$, we have $\chi_1 = \chi_2^{-1}$. The local root number is then defined as
\[
W(\chi, \psi_K) := W(\chi_1, \psi_F) W(\chi_2, \psi_F).
\]
For types (a) and (b), the local root number $W(\chi, \psi_K)$ is as defined in Section \ref{sec:localrt}.

We now define the \emph{relative local root number} $W(\kappa, \chi)$ by
\[
W(\kappa, \chi) := W(\chi, \psi_K)/W(\kappa, \psi_F).
\]
It can be verified that this definition is independent of the choice of $\psi_F$.

Relative local root numbers can be used to compute global root numbers. To illustrate this, suppose $F$ is a number field, $K$ is a quadratic extension of $F$, and $\kappa$ is the quadratic Hecke character of $F$ associated with the extension $K/F$. Let $\chi$ be any unitary Hecke character of $K$ that coincides with $\kappa$ on $F$. By Proposition \ref{prop:anticychar}, this is equivalent to $\chi$ being an anticyclotomic character. Let $w$ be a place of $F$. Let $F_w$ denote the completion of $F$ at $w$, and then $K_w := K \otimes_{F} F_w$ is an $F_w$-algebra of type (a), (b), or (c) as described above, corresponding to whether $w$ ramifies, remains inert, or splits in $K$. We then have the following result:

\begin{proposition}[{\cite[Proposition 6]{MR658544}}] \label{prop:localglobal}
The global root number $W(\chi)$ can be expressed as
\[
W(\chi) = \prod_{w \in \scV_{F}} W(\kappa_w, \chi_w).
\]
\end{proposition}

Using Proposition \ref{prop:localrt}(ii)(a) (see \cite[Proposition 11]{MR658544}), it is easy to see that $W(\kappa_w, \chi_w) = 1$ if $\chi_w$ is unramified. Thus, the product is well-defined.

Therefore, the problem is reduced to computing the relative local root numbers.

\subsection{Local computations I: relative local root numbers}
The computation of relative local root numbers depends on both the decomposition type of places $w$ of $F$ in $K$ and whether $\chi$ is unramified at $w$. The various cases are summarized in the following table.

\begin{table}[H]
\centering
\begin{tabular}{|cc|c|c|}
\hline
\multicolumn{1}{|l|}{}                                         & \multicolumn{1}{l|}{}    & \multicolumn{1}{c|}{$\chi$ unramified at $w$} & \multicolumn{1}{c|}{$\chi$ ramified at $w$} \\ \hline
\multicolumn{1}{|c|}{}                                         & $w$ unramified in $K$    & (A)                           & (E)                                         \\ \cline{2-4} 
\multicolumn{1}{|c|}{\multirow{-2}{*}{Odd finite places $w$}}  & $w$ ramified in $K$      & (B)                      & (C)                         \\ \hline
\multicolumn{1}{|c|}{}                                         & $w$ unramified in $K$    & (A)                           & (F)                                         \\ \cline{2-4} 
\multicolumn{1}{|c|}{\multirow{-2}{*}{Even finite places $w$}} & $w$ ramified in $K$      & (B)                      & (G)                                         \\ \hline
\multicolumn{2}{|c|}{Archimedean places $w$}                                             & \multicolumn{2}{|c|}{$\chi$ of infinite type $(1,0)$: (D)}          \\ \hline
\end{tabular}
\caption{Cases for relative local root numbers}
\label{table}
\end{table}

In the rest of this subsection, let $K/F$ be a quadratic extension of local fields of characteristic $0$ with residue characteristic $p$, or let $K = F \oplus F$. Let $\chi$ be a unitary character $\chi: K^\times \to \CC^\times$ such that $\chi|_{F^\times} = \kappa$, where $\kappa$ is the quadratic character of $F^\times$ associated with the quadratic extension $K/F$.

Most entries in Table \ref{table} have been computed in the literature, primarily based on Proposition \ref{prop:localrt}. The results are summarized as follows:
\begin{itemize}
    \item \textbf{Box (A)}: This case is computed in \cite[Proposition 11]{MR658544} as $W(\kappa, \chi) = 1$.
    \item \textbf{Box (B)}: This case is excluded by Proposition \ref{prop:anticychar}.
    \item \textbf{Box (C)}: This case is computed in \cite[Proposition 2]{MR1133776}:
    \begin{equation}\label{equation:ramified}
    W(\kappa, \chi) = 
    \begin{cases}
    \dleg{2}{p}, &\quad \text{if } f(\chi) = 1, \\
    \dleg{-2l_{\chi}}{p} \chi(\varpi_K^{f-1}) i^\delta, &\quad \text{if } f(\chi) > 1,
    \end{cases}
    \end{equation}
    where $l_{\chi}$ is an invertible residue class modulo $p$ such that $\chi(1 + \varpi_K^{f-1}) = \exp(2\pi i l_{\chi}/p)$, and 
    \[
    \delta = 
    \begin{cases}
    0, &\quad \text{if } p \equiv 1 \pmod{4}, \\
    1, &\quad \text{if } p \equiv 3 \pmod{4}.
    \end{cases}
    \]
    \item \textbf{Box (D)}: This case is computed in \cite[Proposition 12]{MR658544} as $W(\kappa, \chi) = 1$.
\end{itemize}

In this section, we compute the missing box (E). Since boxes (F) and (G) are not needed for our purposes, we will not address these cases in this article, although they are of independent interest.

For box (E), we further divide it into two subcases:
\begin{itemize}
    \item \textbf{Case (E:spl)}: Here, $w$ splits in $K$, corresponding to type (c) introduced earlier in this section.
    \item \textbf{Case (E:inert)}: Here, $w$ remains inert in $K$, corresponding to type (b) introduced earlier in this section.
\end{itemize}

In the split case, we have the following result.
\begin{proposition}\label{prop:splitcase}
In the case (E:spl), $W(\kappa, \chi) = \chi_2(-1)$.
\end{proposition}

\begin{proof}
By definition, $W(\chi, \psi_K) = W(\chi_1, \psi_F) W(\chi_2, \psi_F)$ with $\chi_1 = \chi_2^{-1}$. The proposition then follows from a combination of \cite[(2.66)]{MR2882696} and \cite[(2.82)]{MR2882696}.
\end{proof}

Now consider the inert case, i.e., the case (E:inert). Before presenting the formula, we introduce some additional notations. 

Recall that we have assumed $K = F(\sqrt{-D})$ for some $-D \in F$. Let $\varpi \in F$ be the uniformizer of $F$ so determined, which is also a uniformizer of $K$ since the extension $K/F$ is inert. Denote by $q_K$ and $q_F$ the cardinalities of the residue fields of $K$ and $F$, respectively. We follow the same conventions as in \cite[Section 5]{MR658544}. Let $\dif_F x$ and $\dif_K y$ denote fixed Haar measures on $F$ and $K$, respectively. The restrictions of these measures to $U_F$ and $U_K$ are Haar measures on the respective groups. Define $V = U_K / U_F$, and write $\dif z$ for the quotient measure $\dif_K y / \dif_F x$ on $V$. For a continuous function $f: U_K \to \CC$, we have the relation:
\begin{equation} \label{eq:quointegral}
\int_V \left( \int_{U_F} f(xz) \dif_F x \right) \dif z = \int_{U_K} f(y) \dif_K y.
\end{equation}

For $z \in K$, define
\[
\lambda(z) = 
\begin{cases}
1, & \text{if } \tr_{K/F}(z) \in U_F, \\
0, & \text{otherwise}.
\end{cases}
\]

\begin{proposition}\label{prop:inertcase}
In the case (E:inert), we have
\[
W(\kappa, \chi) = (-1)^f q_K^{f/2} \frac{\Meas(\calO_F)}{\Meas(\calO_K)} \chi^{-1}(\sqrt{-D}) \int_V \lambda\left(\frac{z}{\varpi^f}\right) \dif z,
\]
where $f := f(\chi) \geq 1$. In particular, $W(\kappa, \chi)$ has the same sign as $(-1)^f \chi^{-1}(\sqrt{-D})$.
\end{proposition}

\begin{proof}
This proposition extends Proposition \ref{prop:localrt} and parallels the ramified case, as detailed in \cite[Proposition 1]{MR1133776}. Fix a nontrivial unitary character \(\psi_F: F \to \mathbb{C}^\times\), and let \(\psi_K = \psi_F \circ \tr_{K/F}\). Denote \(m = m(\psi_F)\). Since \(K/F\) is inert, we have \(m(\psi_K) = m\).  

Using Proposition \ref{prop:localrt}(iii) for \(\chi\), we obtain 
\begin{equation}\label{equation:inert1}
\int_{U_K} \chi^{-1}(y) \psi_K\left(\frac{y}{\varpi^{m+f}}\right) \dif_K y = q_K^{-f/2} \Meas(\calO_K) \chi\left(\frac{1}{\varpi^{f+m}}\right) W(\chi, \psi_K).
\end{equation}  
Next, by applying \eqref{eq:quointegral}, the left-hand side can be integrated over \(U_F\) and \(V\)
\[
\int_{U_K} \chi^{-1}(y) \psi_K\left(\frac{y}{\varpi^{m+f}}\right) \dif_K y = \int_{V} \chi^{-1}(z) \left( \int_{U_F} \chi^{-1}(x) \psi_F\left(\tr_{K/F} \left(\frac{zx}{\varpi^{m+f}}\right)\right) \dif_F x \right) \dif z.
\]
Since \(\chi|_{U_F} = \kappa\), \(\kappa^{-1} = \kappa\), and \(\tr_{K/F}\) is \(F\)-linear, this reduces to
\[
\int_{U_K} \chi^{-1}(y) \psi_K\left(\frac{y}{\varpi^{m+f}}\right) \dif_K y = \int_{V} \chi^{-1}(z) \left( \int_{U_F} \kappa(x) \psi_F\left(\tr_{K/F} \left(\frac{z}{\varpi^f}\right) \frac{x}{\varpi^m}\right) \dif_F x \right) \dif z.
\]
Applying Proposition \ref{prop:localrt}(iii) to \(\kappa\), we have
\[
\int_{U_F} \kappa(x) \psi_F\left(\frac{x}{\varpi^m}\right) \dif_F x = \Meas(\calO_F) \kappa\left(\frac{1}{\varpi^m}\right) W(\kappa, \psi_F).
\]
Using the \(\lambda\)-function and substituting \(z' = \tr_{K/F}\left(\frac{z}{\varpi^f}\right)\), this yields
\begin{equation}\label{equation:inert2}
\int_{U_K} \chi^{-1}(y) \psi_K\left(\frac{y}{\varpi^{m+f}}\right) \dif_K y = \Meas(\calO_F) \kappa\left(\frac{1}{\varpi^m}\right) \int_{V} \chi^{-1}(z)  \kappa(z') \lambda\left(\frac{z}{\varpi^f}\right)  \dif z \cdot W(\kappa, \psi_F).
\end{equation}
We note that when \(z' = \tr_{K/F}(\frac{z}{\varpi^f}) \in U_F\), we have \(\kappa(\tr_{K/F}(\frac{z}{\varpi^f})) = 1\), as \(\kappa\) is unramified. Moreover, \(\kappa(\varpi) = \chi(\varpi) = -1\). Comparing both sides of \eqref{equation:inert1} and \eqref{equation:inert2}, these observations allow us to conclude that
\[
W(\kappa, \chi) = (-1)^f q_K^{f/2} \frac{\Meas \calO_F}{\Meas \calO_K} \int_{V} \chi^{-1}(z) \lambda\left(\frac{z}{\varpi^f}\right) \dif z.
\]

Finally, it suffices to consider the remaining integral
\[
\int_{V} \chi^{-1}(z) \lambda\left(\frac{z}{\varpi^f}\right) \dif z.
\]
Let \(z = a + b \sqrt{-D} \in U_{K} \). One checks that the following holds:
\begin{align*}
\tr_{K/F}\left(\frac{z}{\varpi^f}\right) \in U_F 
&\Longleftrightarrow \tr_{K/F}(z) \in \varpi^f U_F \\
&\Longleftrightarrow a \in \varpi^f U_F, \, b \in U_F \\
&\Longleftrightarrow z \in \varpi^f U_F + U_F \sqrt{-D}.    
\end{align*}
Moreover, there exists a bijection
\[
\frac{\varpi^f U_F + U_F \sqrt{-D}}{U_F} \xleftrightarrow{1:1} \{\sqrt{-D} + x \varpi^f \mid x \in U_F\}.
\]
Note that for $z=\sqrt{-D} + x \varpi^f$, we have 
\[\chi^{-1}(z)=\chi^{-1}(\sqrt{-D} + x \varpi^f)=\chi^{-1}(\sqrt{-D})\chi^{-1}  \left(1 + \dfrac{x}{\sqrt{-D}}\varpi^f\right)=\chi^{-1}(\sqrt{-D}).\] 
For the last equality above, we first note that $\sqrt{-D} \in U_{K}$ since $p$ remains inert in $K$, and therefore by the definition of the conductor of $\chi$, we see 
\[
\chi^{-1}(1 + (x/\sqrt{-D}) \varpi^f ) = 1.
\]
Therefore, we have
\begin{align*}
\int_{V} \chi^{-1}(z) \lambda\left(\dfrac{z}{\varpi^f}\right) \dif z &= \int_{V} \chi^{-1}(\sqrt{-D} ) \lambda\left(\dfrac{z}{\varpi^f}\right) \dif z \\
&= \chi^{-1}(\sqrt{-D}) \int_{V} \lambda\left(\frac{z}{\varpi^f}\right) \dif z.    
\end{align*}
The proposition is then proved. 
\end{proof}

\begin{remark}
The main improvement in our computation compared to \cite{MR658544, MR1133776} is that we consider the case where \(\chi\) is ramified at an unramified place \(w\) of \(K\) above \(\mathbb{Q}\), i.e., the box (E). This case is excluded by assumption (iii) on \cite[page 518]{MR658544}. In \textit{loc. cit.}, the relative root number in the inert case is therefore straightforward, namely \(W(\kappa, \chi) = 1\).
\end{remark}

\subsection{Local computations II: the twist of local root numbers}
In this subsection, let \(F = \mathbb{Q}\) and \(K = \mathbb{Q}(\sqrt{-D})\) be an imaginary quadratic field with \(\kappa\) the quadratic character associated with \(K/\QQ\).

\emph{We remark that from this point onward, we are changing notation by letting $v$ be a place of $K$, not of $F$}.

Let \(\chi = \varphi \rho \in \frX_{\varphi}^{\ac}\). In this section, we relate the root numbers \(W(\varphi \rho)\) and \(W(\varphi)\) by computing the quotient
\[
R(\varphi, \rho) := \frac{W(\varphi \rho)}{W(\varphi)} = 
\prod_{\ell \in \scV_{\mathbb{Q}}} \frac{W(\kappa_{\ell}, (\varphi \rho)_{\ell})}{W(\kappa_{\ell}, \varphi_{\ell})} = \prod_{v \in \scV_{K}} \frac{W((\varphi \rho)_{v}, \psi_K)}{W(\varphi_{v}, \psi_K)}.
\]
Here, the second equality follows from Proposition \ref{prop:localglobal}, and the third equality follows from the definition of relative local root numbers. There are several cases of local quotients.

\begin{table}[H]
\centering
\begin{tabular}{|cc|c|c|}
\hline
\multicolumn{2}{|l|}{}                                                              & \(\varphi_v\) is unramified & \(\varphi_v\) is ramified \\ \hline
\multicolumn{2}{|c|}{Archimedean places \(v\)}                   & \multicolumn{2}{|c|}{Infinite type (1,0): (I)} \\ \hline
\multicolumn{2}{|c|}{Finite places \(v\), \(v \nmid p\)}           & (II)                       & (III)                    \\ \hline
\multicolumn{1}{|c|}{}                                    & \(p\) splits in \(K\): \(p\calO_K = v \barv\) & \multicolumn{2}{|c|}{(IV)}                 \\ \cline{2-4} 
\multicolumn{1}{|c|}{}                                    & \(p\) remains inert in \(K\): \(p\calO_K = (p)\) & \multicolumn{2}{|c|}{(V)}               \\ \cline{2-4} 
\multicolumn{1}{|c|}{\multirow{-3}{*}{Places \(v \mid p\)}} & \(p\) ramifies in \(K\): \(p\calO_K = v^2\)     & \multicolumn{2}{|c|}{(VI): see Table \ref{tab:my-table}}                 \\ \hline
\end{tabular}
\caption{The cases of the twist of local root numbers}
\label{table2}
\end{table}

Here, the cases where $v \nmid p$ corresponds to the places where \(\rho\) is unramified, and the cases where $v \mid p$ correspond to those where \(\rho\) is ramified. Before computing \(R(\varphi, \rho)\), note that \(W(\varphi \rho)\) and \(W(\varphi)\) are \(1\) or \(-1\), hence \(R(\varphi, \rho) \in \{\pm 1\}\).

\subsubsection{Local relative root number outside $p$}
In this subsection, we prove the following proposition.
\begin{proposition}\label{prop:reducetop}
    Notations being as above, we have 
    \[
    R(\varphi, \rho) = \dfrac{W(\kappa_p, (\varphi \rho)_p)}{W(\kappa_p, \varphi_p)}.
    \]
    As a result, the quotient is determined by the relative root numbers at \(p\).
\end{proposition}

\begin{proof}
    The strategy is to check every place \(\ell\) of \(\QQ\) outside \(p\). Note that once $\psi_{\QQ}$ is chosen, the additive character $\psi_{K} = \psi_{\QQ} \circ \tr_{K/\QQ}$ is then fixed.

    \textbf{Case (I) in Table \ref{table2}: \(\ell \mid \infty\).} This corresponds to box (D) in Table~\ref{table}. The relative root numbers for \(\varphi \rho\) and \(\rho\) are both \(1\), hence the quotient is \(1\).

    \textbf{Case (II) in Table \ref{table2}: \(\ell \nmid p\) and \(\varphi\) is unramified at \(\ell\).} Since \(\rho\) is also unramified at \(\ell\) (by Proposition~\ref{prop:localrho}), this corresponds to box (A) in Table~\ref{table}. The relative root numbers are \(1\), so the quotient is \(1\).

    \textbf{Case (III) in Table \ref{table2}: \(\ell \nmid p\) and \(\varphi\) is ramified at \(\ell\).} Depending on the decomposition type of $\ell$ in $K$, we have the following cases.

    \underline{(a) The prime $\ell$ does not split in $K$}. In this case, there are only one place $v$ of $K$ lying over $\ell$. In this case, the local root number formula (Proposition~\ref{prop:localrt}(iii)) is more relevant. We have
    \[
    \frac{W(\kappa_{\ell}, (\varphi \rho)_\ell)}{W(\kappa_{\ell}, \varphi_\ell)} = \frac{W(\kappa_{\ell}, (\varphi \rho)_v) W(\kappa_{\ell}, \psi_{\QQ_\ell})}{W(\kappa_{\ell}, \varphi_v) W(\kappa_{\ell}, \psi_{\QQ_\ell})}
    = \frac{W((\varphi \rho)_v, \psi_{K_v})}{W(\varphi_v, \psi_{K_v})},
    \]
    where \(\psi_{\QQ_\ell}\) is a unitary additive character of \(\QQ_\ell^\times\), and \(\psi_{K_v} = \psi_{\QQ_\ell} \circ \tr_{K_v/\QQ_\ell}\). The second equality follows from the definition of the relative root number.

    Since \(\varphi\) is ramified and \(\rho\) is unramified at \(\ell\), the local conductors of \(\varphi\) and \(\varphi \rho\) at \(\ell\) are the same. Hence, applying Proposition~\ref{prop:localrt}(ii), we have 
    \[
    \dfrac{W(\kappa_{\ell}, (\varphi \rho)_\ell)}{W(\kappa_{\ell}, \varphi_\ell)}
    = \dfrac{W((\varphi \rho)_v, \psi_{K_v})}{W(\varphi_v, \psi_{K_v})}
    = \dfrac{(\varphi \rho)_v^{-1}(\beta) \int_{\calO^\times} (\varphi \rho)_v^{-1}(x) \psi_{K_v}(\beta x) \, \dif x}{\varphi_v^{-1}(\beta) \int_{\calO^\times} \varphi_v^{-1}(x) \psi_{K_v}(\beta x) \, \dif x},
    \]
    where \(\beta \in \varpi_v^{-m(\psi_{K_v}) - f(\varphi_{v})} U_{K_v}\). Since \(\rho\) is unramified at \(\ell\), \(\rho_{v}(x) = 1\) for \(x \in U_{K_v}\), and hence the equality reduces to 
    \[
    \frac{W(\kappa_{\ell}, (\varphi \rho)_\ell)}{W(\kappa_{\ell}, \varphi_v)} = \frac{(\varphi \rho)_v^{-1}(\beta)}{\varphi_v^{-1}(\beta)} = \rho_{v}(\varpi_v)^{m(\psi_{K_v}) +  f(\varphi_v)}.
    \]
    Next we show that when $\ell$ does not split in $K$, we always have $\rho_{v}(\varpi_v) = 1$. Indeed, in this case, $\varpi_{v}^{\sfc} = u \varpi_{v}$ for some $u \in U_{K_{v}}$. Then we see that
    \[
    \rho_v(\varpi_{v}^{\sfc}) = \rho_{v}(u)\rho_{v}(\varpi_v) = \rho_{v}(\varpi_v),
    \]
    since $\rho_v$ is an unramified character. Meanwhile, since $\rho_v$ is anticyclotomic, we have
    \[
    \rho_v(\varpi_{v}^{\sfc}) = \rho_v(\varpi_{v})^{-1}.
    \]
    Combining the two equations, we see that $\rho_v(\varpi_{v})^{2} = 1$. Since $\rho_v$ has order $p^{k}$ for an odd prime $p$, this implies that $\rho_v(\varpi_{v}) = 1$, as desired.

    \underline{(b) The prime $\ell$ splits in $K$}. In this case, we write
    \[
    \frac{W(\kappa_{\ell}, (\varphi \rho)_\ell)}{W(\kappa_{\ell}, \varphi_\ell)}
    = \frac{W((\varphi \rho)_v, \psi_{K_v})}{W(\varphi_v, \psi_{K_v})} \cdot \frac{W((\varphi \rho)_{\bar{v}}, \psi_{K_{\bar{v}}})}{W(\varphi_{\bar{v}}, \psi_{K_{\bar{v}}})}.
    \]
    Following the same argument as in the inert and ramified cases, since \(\rho\) is unramified at \(\ell\), using Proposition~\ref{prop:localrt}(iii), we have
    \[
    \frac{W(\kappa_{\ell}, (\varphi \rho)_\ell)}{W(\kappa_{\ell}, \varphi_\ell)} = (\rho_v(\beta) \rho_{\bar{v}}(\beta))^{-1} = 1,
    \]
    for $\beta \in \varpi_{v}^{-m(\psi_{K_v}) - f(\varphi_v)} U_{K_v}$, with the last equality following from \(\rho_v \rho_{\bar{v}} = 1\). 

The proposition is then proved.
\end{proof}

\subsubsection{Local relative root number at $p$}

By Proposition \ref{prop:reducetop}, it then remains to compute the quotient of the relative root numbers at \( p \). We divide this computation into cases based on whether \( p \) splits, remains inert, or ramifies in \( K \), as summarized in Table \ref{table2}.

\begin{lemma}[$p$ splits in $K$]\label{lemma:psplits}
    If \(p\calO_K = v \barv\), then
    \[
    \dfrac{W(\kappa_{p}, (\varphi \rho)_p)}{W(\kappa_{p}, \varphi_p)} = 1.
    \]
\end{lemma}

\begin{proof}
    This corresponds to \textbf{Case (IV)} in Table \ref{table2}. If $\varphi$ is ramified at places  over $p$ , it can be treated using Proposition \ref{prop:splitcase}, which yields
    \[
    \frac{W(\kappa_{p}, (\varphi \rho)_p)}{W(\kappa_{p}, \varphi_p)} = \frac{\varphi_v^{-1}(-1) \rho_v^{-1}(-1)}{\varphi_v^{-1}(-1)} = \rho_v(-1) =  1,
    \]
   where the last equality follows from the fact that $\rho_{v}(-1)^{2} = 1$ and $\rho$ has $p$-power order for an odd prime $p$.

    If $\varphi$ is unramified at places above $p$, then the denominator is equal to $1$, corresponding to \textbf{Case (A)} of Table \ref{table}. Applying Proposition \ref{prop:splitcase} to the numerator, we obtain 
    \[
    \frac{W(\kappa_{p}, (\varphi \rho)_p)}{W(\kappa_{p}, \varphi_p)} = \varphi_v^{-1}(-1) \rho_v^{-1}(-1) = 1.
   \]
   For the last equality, note that $\varphi_v$ is unramified, and hence it is trivial on $U_{K_v}$. Since $-1 \in U_{K_v}$, we then have $\varphi_v(-1) = 1$. Moreover, $\rho_v(-1) = 1$ following the same reasoning in the previous paragraph. This completes the proof of the lemma.
\end{proof}

\begin{lemma}[$p$ remains inert in $K$]\label{lemma:pinert}
    If \( v \mid p \) is inert in \( K \), denote $N_{\inert, \varphi} := j + f(\varphi_p) - 1$, then
    \[
    \dfrac{W(\kappa_{p}, (\varphi \rho)_p)}{W(\kappa_{p}, \varphi_p)} = \begin{cases}
        1, &\quad n < N_{\inert, \varphi},  \\ 
        (-1)^{f(\varphi_v \rho_v) - f(\varphi_v)}, &\quad n = N_{\inert, \varphi}, \\ 
        (-1)^{n-j+1-f(\varphi_v)}, &\quad n > N_{\inert, \varphi}.
    \end{cases}
    \]
\end{lemma}

\begin{proof}
    This corresponds to \textbf{Case (V)} in Table \ref{table2}. If $\varphi$ is ramified at $v$, it can be treated using Proposition \ref{prop:inertcase}, which yields
    \begin{align*}
        \dfrac{W(\kappa_{p}, (\varphi \rho)_p)}{W(\kappa_{p}, \varphi_p)} &= c \cdot \dfrac{(-1)^{f(\varphi_v \rho_v)} \varphi_v^{-1}(\sqrt{-D}) \rho_v^{-1}(\sqrt{-D})}{(-1)^{f(\varphi_v)} \varphi_v^{-1}(\sqrt{-D})} \\ 
        &= c \cdot (-1)^{f(\varphi_v \rho_v) - f(\varphi_v)} \cdot \rho_v(\sqrt{-D})^{-1}    
    \end{align*}
    for some constant $c > 0$. By Proposition \ref{prop:localrho}(ii), $\rho_v(\sqrt{-D}) = 1$. Moreover, the quotient \( \frac{W(\kappa_{p}, (\varphi \rho)_p)}{W(\kappa_{p}, \varphi_p)} \) belongs to \( \{ \pm 1 \} \), which implies \( c = 1 \).

    If $\varphi$ is unramified at $v$, the denominator is $1$ corresponding to \textbf{Case (A)} in Table \ref{table}. Then apply Proposition \ref{prop:splitcase} to the numerator and we have 
    \begin{align*}
    \dfrac{W(\kappa_{p}, (\varphi \rho)_p)}{W(\kappa_{p}, \varphi_p)} &= c \cdot (-1)^{f(\varphi_v \rho_v)} \varphi_v^{-1}(\sqrt{-D}) \rho_v^{-1}(\sqrt{-D}) \\ 
    &= c \cdot (-1)^{f(\varphi_v \rho_v)} \cdot \rho_v(\sqrt{-D})^{-1}    
    \end{align*}
    for some constant $c > 0$. By the same argument in the previous paragraph, we see $\rho_v(\sqrt{-D}) = 1$ and $c=1$. Therefore, the first part of the lemma is proved. Now, we compute $f(\varphi_v \rho_v)$ using Proposition \ref{prop:localrho}(2):
    
    \textbf{Case 1: $n < j + f(\varphi_v) - 1$}.
    In this case, $0 \leq f(\rho_v) < f(\varphi_v)$.
    Thus, $f(\varphi_v \rho_v) = f(\varphi_v)$, so $f(\varphi_v \rho_v) - f(\varphi_v) = 0$. In this case, 
    \[
    \dfrac{W(\kappa_{p}, (\varphi \rho)_p)}{W(\kappa_{p}, \varphi_p)} = 1.
    \]

    \textbf{Case 2: $n = j + f(\varphi_v) - 1$}. In this case, we do not compute any further.

    \textbf{Case 3: $n > j + f(\varphi_v) - 1$}. In this case, $f(\rho_v) = n - j + 1 > f(\varphi_v)$. Hence $f(\varphi_v \rho_v) = f(\rho_v)$, so 
    \[
    \dfrac{W(\kappa_{p}, (\varphi \rho)_v)}{W(\kappa_{p}, \varphi_p)} = (-1)^{n - j + 1 - f(\varphi_v)}.
    \]
    This proves the lemma.
\end{proof}

The most complicated case occurs when $p$ ramifies in $K$ as \( p\calO_{K} = v^2 \). In this case, \( \varpi_{v} = \sqrt{-D} \) serves as the uniformizer of \( K_v \) since \( p \mid -D \). Before explicit computations, we record the following observation.

\begin{proposition} \label{prop:jacobisumobservation}
Let $F$ be a nonarchimedean local field with odd residual characteristic $p$. Let $K/F$ be a ramified quadratic extension, and let $\kappa: F^\times \to \{\pm 1\}$ be the quadratic character associated to $K/F$. If $\chi: K^\times \to \CC^\times$ is any character such that $\chi|_{F^\times} = \kappa$, then $f(\chi)$ is either $1$ or an even integer $n \geq 2$.
\end{proposition}

\begin{proof}
This is an observation from \cite[page 160]{MR1133776}, where $\chi$ is assumed to be a \emph{unitary character}. We revisit the proof in \textit{loc. cit.} and show that this assumption is in fact unnecessary.

By local class field theory, since $K/F$ is a ramified quadratic extension and $p$ is odd, $f(\kappa) = 1$. Since $\chi|_{F^\times} = \kappa$, the restriction of $\chi$ to $U_{F}$ cannot be trivial. Given that $U_F \subseteq U_K$, it follows that $\chi$ is nontrivial on $U_K$. By definition of the conductor of a character, we must have $f(\chi) \geq 1$.

Suppose for the sake of contradiction that $f(\chi) > 1$ and $f(\chi)$ is an odd integer. We can therefore write $f(\chi) = 2m + 1$ for some integer $m \geq 1$, which means that $\chi$ is trivial on $U_K^{(2m+1)} = 1 + \varpi_K^{2m+1}\calO_K$, but nontrivial on $U_K^{(2m)} = 1 + \varpi_K^{2m} \calO_K$. Consequently, $\chi$ induces a nontrivial character $\barchi$ on the quotient group
\[
\barchi: U_K^{(2m)} / U_K^{(2m+1)} \longrightarrow \mathbb{C}^\times.
\]

Because the extension $K/F$ is ramified, the residue field extension is trivial, which means $\calO_K/\varpi_K \calO_K \cong \calO_F/\varpi_F \calO_F$. The ramification index is $e(K/F) = 2$, so the valuations satisfy $v_K(\varpi_F) = 2$. Thus, we can choose uniformizers such that $\varpi_F = u \varpi_K^2$ for some unit $u \in U_F$. Now consider elements in $U_F^{(m)} = 1 + \varpi_F^m \mathcal{O}_F$. For any $x \in U_F$, we can write
\[
1 + x \varpi_F^m = 1 + x (u \varpi_K^2)^m = 1 + x u^m \varpi_K^{2m}.
\]
As $x$ ranges over a complete set of representatives of the residue field $\calO_F/\varpi_F \calO_F$, the term $x u^m$ also ranges over a complete set of representatives of $\calO_K/\varpi_K \calO_K$ since $u^m$ is a unit. Therefore, the natural homomorphism
\[
U_F^{(m)} \longrightarrow U_K^{(2m)} / U_K^{(2m+1)}
\]
is surjective. Recall that $\chi$ is nontrivial on this quotient, there must exist some $x \in U_F$ such that $\chi(1 + x \varpi_F^m) \neq 1$. On the other hand, we can evaluate this value directly using the restriction property $\chi|_{F^\times} = \kappa$ as
\[
\chi(1 + x \varpi_F^m) = \kappa(1 + x \varpi_F^m).
\]
By our assumption, $f(\chi) > 1$, which forces $m \geq 1$. This implies that the element $1 + x \varpi_F^m$ belongs to $U_F^{(1)}$. However, $f(\kappa) = 1$, meaning $\kappa$ is trivial on $U_F^{(1)}$. Thus $\kappa(1 + x \varpi_F^m) = 1$. This yields $\chi(1 + x \varpi_F^m) = 1$, which contradicts our previous deduction that $\chi(1 + x \varpi_F^m) \neq 1$. 

Hence, the conductor exponent $f(\chi)$ cannot be an odd integer strictly greater than 1. If $f(\chi) \neq 1$, it must be an even integer.
\end{proof}

For simplicity, we write \( f_1 = f(\rho_v \varphi_v) \) and \( f_2 = f(\varphi_v) \), and we denote \( l_1 = l_{\rho_v \varphi_v} \) and \( l_2 = l_{\varphi_v} \). We note that $l_i$ is defined using the fixed uniformizer $\varpi= \sqrt{-D}$. We also note that:
\begin{itemize}
    \item If $f(\rho_v) \neq f(\varphi_v)$, then $f(\rho_v\varphi_v) = \max \{f(\rho_v), f(\varphi_v) \}$. 
    \item The exceptional case $f(\rho_v)=f(\varphi_v)$ occurs at a unique layer $n=N_{\ram,\varphi}:=j+\frac{f(\varphi_v)}{2}$ by Proposition \ref{prop:localrho}. Note that by Proposition \ref{prop:anticychar}, the local character $\varphi_v$ is ramified, and hence $N_{\ram, \varphi} > j$. In this situation, $f(\rho_v\varphi_v)\leq f(\rho_v)=f(\varphi_v)$, and strict inequality may occur. We shall treat this case separately.
\end{itemize}

\begin{lemma}[$p$ ramifies in $K$]\label{lemma:pramified}
    Suppose $p$ ramifies in $K$ as $p\calO_{K}=v^2$ and let $\rho \in \frY_{n}^{\dagger}$. Then if $n \leq j$, we have 
    \[
    \dfrac{W(\kappa_p,(\varphi \rho)_p)}{W(\kappa_p,\varphi_p)} = 1.
    \] 
    If $n > j$, and futher suppose that $n \neq N_{\ram, \varphi}$, we have the following cases:
\begin{enumerate}[\rm (i)]
    \item When $p \equiv 1 \pmod 4$, then 
    \[
    \dfrac{W(\kappa_{p},(\varphi \rho)_p)}{W(\kappa_{p},\varphi_p)} = \begin{cases}
\dleg{l_1 l_2}{p}, &\quad\text{if } f(\rho_v) > f(\varphi_v)>1 ,\\
1, &\quad \text{if } f(\rho_v)<f(\varphi_v) \text{ and } f(\varphi_v)>1, \\
\dleg{l_1}{p} \varphi_v(\varpi_{v}), &\quad\text{if } f(\rho_v) > f(\varphi_v)=1.\\
\end{cases}
\]
\item When $p \equiv 3 \pmod 4$, then
    \begin{equation*}
\dfrac{W(\kappa_{p},(\varphi \rho)_p)}{W(\kappa_{p},\varphi_p)} = \begin{cases}
\dleg{l_1 l_2}{p}(-1)^{(f(\rho_v)-f(\varphi_v))/2}, &\quad\text{if } f(\rho_v) > f(\varphi_v)>1,\\
1, &\quad\text{if } f(\rho_v)<f(\varphi_v) \text{ and } f(\varphi_v)>1,\\
\dleg{l_1}{p} (-1)^{\frac{f(\rho_v)}{2}+1} \dfrac{i}{\varphi_v(\varpi_{v})}, &\quad\text{if } f(\rho_v)>f(\varphi_v)=1.
\end{cases}
\end{equation*}
\end{enumerate}
Suppose that $n = N_{\ram, \varphi}$, then
\begin{enumerate}
    \item[\rm (iii)] When $p \equiv 1 \pmod{4}$, then
    \begin{equation*}
        \dfrac{W(\kappa_{p},(\varphi \rho)_p)}{W(\kappa_{p},\varphi_p)} = \begin{cases}
        \dleg{l_2}{p} \varphi_{v}(\varpi_{v}), &\quad f(\varphi_v \rho_v) = 1, \\ 
        \dleg{l_1 l_2}{p}, &\quad f(\varphi_v \rho_v) > 1.
        \end{cases}
    \end{equation*}
    \item[\rm (iv)] When $p \equiv 3 \pmod{4}$, then
    \begin{equation*}
        \dfrac{W(\kappa_{p},(\varphi \rho)_p)}{W(\kappa_{p},\varphi_p)} = \begin{cases}
        (-1)^{f(\rho_v)/2 + 1} \dleg{l_2}{p} \dfrac{\varphi_{v}(\varpi_{v})}{i}, &\quad f(\varphi_v \rho_v) = 1, \\ 
        (-1)^{(f(\rho_v\varphi_v) - f(\rho_v))/2} \dleg{l_1 l_2}{p}, &\quad f(\varphi_v \rho_v) > 1.
        \end{cases}
    \end{equation*}
\end{enumerate}
\end{lemma}

\begin{proof}
Firstly, recall from Proposition \ref{prop:localtotram} that when \( n \leq j \), we have \( K_{n,v_n}^{\ac} = K_v \). Thus, if \( \rho \) has the level \( n \leq j \), then \( \rho_v = 1 \). In this case, 
\[
\dfrac{W(\kappa_{p}, (\varphi \rho)_p)}{W(\kappa_{p}, \varphi_p)} = 1.
\]
Therefore, we only consider the case where \( f(\rho_v) > 0 \). Considering the formula \eqref{equation:ramified}, we have the following cases when $p$ ramifies in $K$:
\begin{table}[H]
    \centering
    \begin{tabular}{|c|c|c|}
    \hline
    & \( p \equiv 1 \pmod{4} \) & \( p \equiv 3 \pmod{4} \) \\ \hline
    \multirow{3}{*}{\( f(\varphi_v) > 1 \)} & \( f(\rho_v) > f(\varphi_v) \) & \( f(\rho_v) > f(\varphi_v) \) \\ \cline{2-3} 
     & \( f(\rho_v) = f(\varphi_v) \) & \( f(\rho_v) = f(\varphi_v) \) \\ \cline{2-3} 
     & \( f(\rho_v) < f(\varphi_v) \) & \( f(\rho_v) < f(\varphi_v) \) \\ \hline
    \( f(\varphi_v) = 1 \) & \( f(\rho_v) > f(\varphi_v) \) & \( f(\rho_v) > f(\varphi_v) \) \\ \hline
    \end{tabular}
    \caption{Cases when \( p \) ramifies in \( K \)}
    \label{tab:my-table}
\end{table}
We consider each case separately.

\textbf{Case 1: $p \equiv 1 \pmod{4}$ and $f(\rho_v) > f(\varphi_v) > 1$.}  
In this case, $f_1 = f(\rho_v)$ and $f_2 = f(\varphi_v)$ are both even. Then using formula \eqref{equation:ramified}, we have
\begin{align*}
 \dfrac{W(\kappa_{p}, (\varphi \rho)_p)}{W(\kappa_{p}, \varphi_p)} = \dfrac{\leg{-2l_1}{p} (\varphi_v \rho_v)(\varpi_{v}^{f_1-1})}{\leg{-2l_2}{p} \varphi_v(\varpi_{v}^{f_2-1})} &= \leg{l_1 l_2}{p} \varphi_v(-D)^{(f_1-f_2)/2} \\ 
 &= \leg{l_1 l_2}{p} \kappa_p(-D)^{(f_1-f_2)/2}.    
\end{align*}
Here we have used the observations that $\varphi_v(-D) = \kappa_{p}(-D)$ and $\rho_v(\varpi) = 1$ by Proposition \ref{prop:localrho}(ii). To compute $\kappa_p(-D)$, since $p$ ramifies in $K$, we write
\[
\kappa_p(-D) = \kappa_p(-1) \kappa_p(p) \kappa_p(D/p) = \kappa_p(p) \kappa_p(D/p),
\]
since $\kappa_p(-1) = 1$ in this case of $p \equiv 1 \bmod{4}$.
Since $K_v/F_p$ is a ramified extension, we need to discuss $K_v = \mathbb{Q}_p(\sqrt{p})$ or $K_v = \mathbb{Q}_p(\sqrt{pr})$, where $r$ is a quadratic non-residue modulo $p$. In both cases, \[
\Norm(K_v^{\times}) = (D)^\ZZ \times \Norm(U_K).
\]  
We have the two subcases individually to prove that $\kappa_p(-D) = 1$.
\begin{enumerate}[(a)]
    \item \underline{$K_v = \QQ_p(\sqrt{p})$, i.e. $-D/p$ is a quadratic residue modulo $p$}. Then since $p \equiv 1 \bmod 4$, the integer $D/p$ is a quadratic residue modulo $p$, and hence $\kappa_p(D/p) = 1$. Moreover, $p$ is in the norm group of $K_v^{\times}$, and hence $\kappa_p(p) = 1$. Therefore, $\kappa_p(-D) = 1$.
    \item \underline{$K_v = \QQ_p(\sqrt{pr})$, i.e. $-D/p$ is a quadratic non-residue modulo $p$}. Then since $p \equiv 1 \bmod 4$, the integer $D/p$ is a quadratic non-residue modulo $p$, and hence $\kappa_p(D/p) = -1$. Moreover, $\kappa_p(p) = -1$ since $p$ does not belong to the norm group of $K_v^{\times}$ anymore. To sum up, we have $\kappa_p(-D) = 1$.
\end{enumerate}
Therefore, in this case,  
\[
\dfrac{W(\kappa_{p}, (\varphi \rho)_p)}{W(\kappa_{p}, \varphi_p)} = \leg{l_1 l_2}{p}.
\]  

\textbf{Case 2: $p \equiv 1 \pmod{4}$ and $f(\rho_v) = f(\varphi_v) > 1$.}  
In this case, $f_2 = f(\varphi_v)$ is even, whereas $f_1=f(\varphi_v\rho_v)$ may be either 1 or a positive even integer by Proposition \ref{prop:jacobisumobservation}. We therefore consider the following two subcases.

\underline{Case 2.1: $f_1=f(\varphi_v\rho_v)=1$.}
Using formula \eqref{equation:ramified}, and noting that different instances of the formula should be applied to the numerator and denominator, we obtain
\begin{align*}
\dfrac{W(\kappa_{p}, (\varphi \rho)_p)}{W(\kappa_{p}, \varphi_p)} = \dfrac{\leg{2}{p} }{\leg{-2l_2}{p} \varphi_v(\varpi_{v}^{f_2-1})} &=\leg{-1}{p} \leg{l_2}{p} \varphi_v(\varpi_v^{-f_2+1}) \\ 
&= \leg{l_2}{p} \varphi_v(\varpi_v^{-f_2+1}),
\end{align*}
where the last equality follows from the assumption that $p \equiv 1 \pmod{4}$, which implies $\leg{-1}{p}=1$. Furthermore,
\[
\varphi_v(\varpi_v)^2=\varphi_v(-D)=1,
\]
so that $\varphi_v(\varpi_v)\in\{\pm1\}$. Since $1-f_2$ is odd, it follows that
\[
\varphi_v(\varpi_v^{-f_2+1})=\varphi_v(\varpi_v).
\]
Therefore,
\[
\dfrac{W(\kappa_{p}, (\varphi \rho)_p)}{W(\kappa_{p}, \varphi_p)} = \leg{l_2}{p} \varphi_v(\varpi_v)
\]

\underline{Case 2.2: $f_1=f(\varphi_v\rho_v)$ is an even integer.} Using formula \eqref{equation:ramified}, 
\begin{align*}
\dfrac{W(\kappa_{p}, (\varphi \rho)_p)}{W(\kappa_{p}, \varphi_p)} = \dfrac{\leg{-2l_1}{p} (\varphi_v \rho_v)(\varpi_{v}^{f_1-1})}{\leg{-2l_2}{p} \varphi_v(\varpi_{v}^{f_2-1})} &= \leg{l_1 l_2}{p} \varphi_v(\varpi_{v}^{f_1-f_2})  \rho_v(\varpi_{v}^{f_1-1}) \\ &=  \leg{l_1 l_2}{p} \varphi_v(-D)^{(f_1-f_2)/2}  \rho_v(\varpi_{v}^{f_1-1}) \\
&= \leg{l_1 l_2}{p}. 
\end{align*}
In the last equality, we used Proposition \ref{prop:localrho}(ii) stating that $\rho_v(\varpi_{v}) = 1$, and the computation in \textbf{Case 1} that $\varphi_v(-D) = 1$.

\textbf{Case 3: $p \equiv 1 \pmod{4}$, $f(\rho_v) < f(\varphi_v)$, and $f(\varphi_v) > 1$.}  
Here, $f_1 = f_2 = f(\varphi_v)$ is even, and $l_1 = l_2$ by definition. Thus, by formula \eqref{equation:ramified} again, 
\[
\dfrac{W(\kappa_{p}, (\varphi \rho)_p)}{W(\kappa_{p}, \varphi_p)} = \dfrac{\leg{-2l_1}{p} (\varphi_v \rho_v)(\varpi_{v}^{f_1-1})}{\leg{-2l_2}{p} \varphi_v(\varpi_{v}^{f_2-1})} = \leg{l_1 l_2}{p} \rho_v(\varpi_{v}^{f_1-1}) = 1.
\]
In the last equality, we used Proposition \ref{prop:localrho}(ii) stating that $\rho_v(\varpi_{v}) = 1$.

\textbf{Case 4: $p \equiv 1 \pmod{4}$ and $f(\rho_v) > f(\varphi_v) = 1$.}  
Using different formulas in \eqref{equation:ramified} for $f_2 = f(\varphi_v) = 1$ and $f_1 = f(\varphi_v \rho_v) = f(\rho_v) > 1$, which is even, we have  
\begin{align*}
  \dfrac{W(\kappa_{p}, (\varphi \rho)_p)}{W(\kappa_{p}, \varphi_p)} &= \dfrac{\leg{-2l_1}{p} (\varphi_v \rho_v)(\varpi_{v}^{f_1-1})}{\leg{2}{p}} \\ 
  &= \leg{-1}{p} \leg{l_1}{p} \varphi_v(\varpi_{v})^{f_1 - 1} \\ 
  &\xlongequal{p \equiv 1 \bmod{4}} \leg{l_1}{p} \varphi_v(\varpi_{v})^{f_1 - 1}.  
\end{align*}
Note that $\varphi_v(\varpi_{v})^2 = \kappa_p(-D) = 1$ by invoking the computation in \textbf{Case 1}, the above formula further simplifies to 
\[
\dfrac{W(\kappa_{p}, (\varphi \rho)_p)}{W(\kappa_{p}, \varphi_p)} = \leg{l_1}{p} \varphi_v(\varpi_{v}).
\] 
By the way, $\varphi_v(\varpi_{v}) \in \{\pm 1\}$ since $\varphi_v(\varpi_{v})^2 = 1$.

\textbf{Case 5: $p \equiv 3 \pmod{4}$ and $f(\rho_v) > f(\varphi_v) > 1$.}  
In this case, $f_1 = f(\rho_v)$ and $f_2 = f(\varphi_v)$ are both even. Then by formula \eqref{equation:ramified}, 
\[
\dfrac{W(\kappa_{p}, (\varphi \rho)_p)}{W(\kappa_{p}, \varphi_p)} = \dfrac{\leg{-2l_1}{p} (\varphi_v \rho_v)(\varpi_{v}^{f_1-1})}{\leg{-2l_2}{p} \varphi_v(\varpi_{v}^{f_2-1})} = \leg{l_1 l_2}{p} \varphi_v(-D)^{(f_1-f_2)/2}.
\]  
Similar to the discussion in \textbf{Case 1}, we find that $\varphi_v(-D) = \kappa_p(-D) = -1$. Hence  
\[
\dfrac{W(\kappa_{p}, (\varphi \rho)_p)}{W(\kappa_{p}, \varphi_p)} = \leg{l_1 l_2}{p} (-1)^{(f_1-f_2)/2}.
\]  

\textbf{Case 6: $p \equiv 3 \pmod{4}$ and $f(\rho_v) = f(\varphi_v) > 1$.}  
This case is similar to \textbf{Case 2}. In this case, $f_2 = f(\varphi_v)$ is even, whereas $f_1=f(\varphi_v\rho_v)$ may be either 1 or a positive even integer by Proposition \ref{prop:jacobisumobservation}. We therefore consider the following two subcases.

\underline{Case 6.1: $f_1=f(\varphi_v\rho_v)=1$.}
Using formula \eqref{equation:ramified}, and noting that different instances of the formula should be applied to the numerator and denominator, we obtain
\begin{align*}
\dfrac{W(\kappa_{p}, (\varphi \rho)_p)}{W(\kappa_{p}, \varphi_p)} 
&= \dfrac{\leg{2}{p}}{\leg{-2l_2}{p}\varphi_v(\varpi_v^{f_2-1})i} \\
&= \leg{-1}{p}\leg{l_2}{p}\varphi_v(\varpi_v^{-f_2+1})/i \\
&= -\leg{l_2}{p}\varphi_v(\varpi_v)^{1-f_2}/i,
\end{align*}
where we used that $p \equiv 3 \pmod{4}$, so that $\leg{-1}{p}=-1$. Note that
\[
\varphi_v(\varpi_v)^2=\varphi_v(-D)=-1,
\]
hence $\varphi_v(\varpi_v)\in\{\pm i\}$. Since $f_2$ is even, we have
\[
\varphi_v(\varpi_v)^{1-f_2} = (-1)^{f_2/2}\varphi_v(\varpi_v).
\]
Therefore,
\[
\dfrac{W(\kappa_{p}, (\varphi \rho)_p)}{W(\kappa_{p}, \varphi_p)} 
= (-1)^{f_2/2+1}\leg{l_2}{p}\frac{\varphi_v(\varpi_v)}{i}.
\].

\underline{Case 6.2: $f_1=f(\varphi_v\rho_v)$ is even positive integer.} Using formula \eqref{equation:ramified}, 
\begin{align*}
\dfrac{W(\kappa_{p}, (\varphi \rho)_p)}{W(\kappa_{p}, \varphi_p)} = \dfrac{\leg{-2l_1}{p} (\varphi_v \rho_v)(\varpi_{v}^{f_1-1})}{\leg{-2l_2}{p} \varphi_v(\varpi_{v}^{f_2-1})} &= \leg{l_1 l_2}{p} \varphi_v(\varpi_{v}^{f_1-f_2})  \rho_v(\varpi_{v}^{f_1-1}) \\ &=  \leg{l_1 l_2}{p} \varphi_v(-D)^{(f_1-f_2)/2}  \rho_v(\varpi_{v}^{f_1-1}) \\
&= (-1)^{(f_1-f_2)/2}\leg{l_1 l_2}{p}.    
\end{align*}
In the last equality, we used Proposition \ref{prop:localrho}(ii) stating that $\rho_v(\varpi_{v}) = 1$, and the computation in \textbf{Case 5} that $\varphi_v(-D) = - 1$.

\textbf{Case 7: $p \equiv 3 \pmod{4}$, $f(\rho_v) < f(\varphi_v)$, and $f(\varphi_v) > 1$.}  
Here, $f_1 = f_2 = f(\varphi_v)$ is even, and $l_1 = l_2$ by definition. Thus, formula \eqref{equation:ramified} gives
\[
\dfrac{W(\kappa_{p}, (\varphi \rho)_p)}{W(\kappa_{p}, \varphi_p)} = \dfrac{\leg{-2l_1}{p} (\varphi_v \rho_v)(\varpi_{v}^{f_1-1})}{\leg{-2l_2}{p} \varphi_v(\varpi_{v}^{f_2-1})} = \leg{l_1 l_2}{p} \rho_v(\varpi_{v}^{f_1-1}) = 1
\]  

\textbf{Case 8: $p \equiv 3 \pmod{4}$ and $f(\rho_v) > f(\varphi_v) = 1$.}  
Using different formulas in \eqref{equation:ramified} for $f_2 = f(\varphi_v) = 1$ and $f_1 = f(\varphi_v \rho_v) = f(\rho_v) > 1$, which is even, we have, by using formula \eqref{equation:ramified}, 
\begin{align*}
\dfrac{W(\kappa_{p}, (\varphi \rho)_p)}{W(\kappa_{p}, \varphi_p)} &= \dfrac{\leg{-2l_1}{p} (\varphi_v \rho_v)(\varpi_{v}^{f_1-1}) i}{\leg{2}{p}} \\ 
&= \leg{-1}{p} \leg{l_1}{p} \varphi_v(\varpi_{v}^{f_1-1}) i \\ &\xlongequal{p \equiv 3 \bmod{4}} -\leg{l_1}{p} \varphi_v(\varpi_{v}^{f_1-1}) i.
\end{align*}
We invoke the computation in \textbf{Case 5} to see that $\varphi_v(\varpi_{v})^2 = \kappa_p(-D) = -1$, and hence
\[
\dfrac{W(\kappa_{p}, (\varphi \rho)_p)}{W(\kappa_{p}, \varphi_p)} =  \leg{l_1}{p} (-1)^{\frac{f_1}{2}+1} \dfrac{i}{\varphi_v(\varpi_{v})}.
\]
By the way, $\varphi_v(\varpi_{v}) \in \{\pm i\}$ since $\varphi_v(\varpi_{v})^2 = -1$. Hence the quotient $i/\varphi_v(\varpi_{v}) \in \{ \pm 1\}$ indeed.
\end{proof}

\begin{remark}
One could consider the general case: let $\chi$ and $\chi_0$ be two Hecke characters of $K$, and compute the quotient 
\[
R(\chi_0, \chi) := \frac{W(\chi_0 \chi)}{W(\chi_0)},
\]
using similar computations. However, this is more complicated since there are more cases to consider, especially the local quotient at the places above $2$. In our case, $\rho$ consists of Hecke characters factoring through the anticyclotomic tower $K_{\infty}^{\ac}/K$, and hence are unramified outside the odd prime $p$, which simplifies the computation significantly.
\end{remark}

\subsection{Global computation}
We summarize our local results to obtain the following theorem on global root numbers. Recall that when $p$ ramifies in $K$ as $p \calO_{K} = v^{2}$, we use $\varpi_{v} = \sqrt{-D}$ as a fixed uniformizer of $\calO_{K_{v}}$.

\begin{theorem} \label{thm:rootnumber}
For $\chi = \varphi \rho \in \frX_{\varphi, n}^{\ac, \dagger}$, the root number $W(\chi)$ is given by the following formulas.
\begin{enumerate}[\rm (i)]
    \item If $p$ splits in $K$, then $W(\chi) = W(\varphi)$,
    \item If $p$ remains inert in $K$, then
    \[
    W(\chi) = \begin{cases}
    W(\varphi), & \text{if } n < j + f(\varphi_p) - 1, \\
    (-1)^{f(\varphi_v \rho_v) - f(\varphi_v)} W(\varphi), & \text{if } n = N_{\inert, \varphi} = j + f(\varphi_v) - 1, \\
    (-1)^{n - j + 1 - f(\varphi_v)} W(\varphi), & \text{if } n > j + f(\varphi_v) - 1.
    \end{cases}
    \]
    \item If $p$ ramifies in $K$ and $n \leq j$, then $W(\chi) = W(\varphi)$.
    \item If $p$ ramifies in $K$ and $p \equiv 1 \pmod{4}$, suppose that $n \neq N_{\ram, \varphi}$, then
    \[
    W(\chi) = \begin{cases}
    W(\varphi) \dleg{l_1 l_2}{p}, & \text{if } 2(n - j) > f(\varphi_v) > 1, \\
    1, &\text{if } 2(n - j) < f(\varphi_v) \text{ and } f(\varphi_v) > 1,\\
    W(\varphi) \dleg{l_1}{p} \varphi_v(\varpi_{v}), & \text{if } 2(n - j) > f(\varphi_v) = 1. 
    \end{cases}
    \]
    \item If $p$ ramifies in $K$ and $p \equiv 3 \pmod{4}$, suppose that $n \neq N_{\ram, \varphi}$, then
    \[
    W(\chi) = \begin{cases}
        W(\varphi) \dleg{l_1 l_2}{p} (-1)^{n-j +(f(\varphi_v)/2)}, & \text{if } 2(n - j) > f(\varphi_v) > 1, \\
        W(\varphi), & \text{if } 2(n - j) < f(\varphi_v) \text{ and } f(\varphi_v) > 1, \\
        W(\varphi) \dleg{l_1}{p} (-1)^{n-j+1} \dfrac{i}{\varphi_v(\varpi_{v})}, & \text{if } 2(n - j) > f(\varphi_v) = 1.
    \end{cases}
    \]    
\end{enumerate}
Additionally, suppose $p$ ramifies in $K$ and $n = N_{\ram, \varphi}$, then
    \begin{equation*}
        W(\chi) = \begin{cases}
        W(\varphi) \dleg{l_2}{p} \varphi_{v}(\varpi_{v}), &\quad f(\varphi_v \rho_v) = 1, \, p \equiv 1 \pmod{4} \\ 
        W(\varphi) \dleg{l_1 l_2}{p}, &\quad f(\varphi_v \rho_v) > 1, \, p \equiv 1 \pmod{4}, \\ 
        W(\varphi) (-1)^{f(\rho_v)/2 + 1} \dleg{l_2}{p} \dfrac{\varphi_{v}(\varpi_{v})}{i}, &\quad f(\varphi_v \rho_v) = 1, \, p \equiv 3 \pmod{4} \\ 
        (-1)^{(f(\rho_v\varphi_v) - f(\rho_v))/2} \dleg{l_1 l_2}{p}, &\quad f(\varphi_v \rho_v) > 1, \, p \equiv 3 \pmod{4}.
        \end{cases}
    \end{equation*}
\end{theorem}

\begin{proof}
By Proposition \ref{prop:reducetop}, we only need to consider the quotient of the relative root number at $p$, which follows from Lemma \ref{lemma:psplits} (when $p$ splits in $K$), Lemma \ref{lemma:pinert} (when $p$ remains inert in $K$), and Lemma \ref{lemma:pramified} (when $p$ ramifies in $K$) for each case in Table \ref{table2}.
\end{proof}

\begin{remark} \label{rmk:greenberg}
This result was first observed by Greenberg in \cite[page 247]{MR700770} when $p$ is unramified in $K$, in the case of elliptic curves $E$. His proof is entirely different from ours, relying on the following fact of Weil \cite[page 161]{weil1971DirichletSeriesAutomorphic}: Let $\chi_1$ and $\chi_2$ be unitary Hecke characters with relatively prime conductors $\frf_{1}$ and $\frf_{2}$ with some extra hypotheses, then 
    \[
    W(\chi_1 \chi_2) = W(\chi_1) W(\chi_2) \chi_1(\frf_{2}) \chi_2(\frf_{1}).
    \]
This fact can only handle the case when $\varphi$ and $\rho$ have coprime conductors. Since $\rho$ is ramified at $p$, it can only handle the case when $\varphi$ is unramified at $p$. From Proposition \ref{prop:anticychar}, this restricts us to the case when $p$ is unramified in $K$ (i.e., the elliptic curve has good reduction at $p$). Our approach of computing local root numbers covers the case when $p$ is ramified.
\end{remark}

\section{Proof of the main result} \label{sec:mainresult}
In this section, we utilize the analytic results on root numbers obtained in the previous sections to study the growth of the Mordell-Weil rank of the abelian variety $A_{\varphi}$ along the anticyclotomic $\ZZ_p$-extension $K_{\infty}^{\ac}/K$. An important tool connecting the analytic properties to the arithmetic side is the well-known Gross-Zagier-Kolyvagin theorem. Yet it requires more work on finding the exact abelian varieties to apply that, which we shall explain now.

Let $\varphi: \AA_{K}^{\times} \rightarrow \CC^{\times}$ be an anticyclotomic Hecke character over $K$ of infinite type $(1,0)$, as described in the previous sections. We consider $\chi := \varphi \rho \in \frX^{\ac}_{\varphi, n}$, where $\rho$ is a Galois character of level $n$. We define $\Phi_{\rho}$ to be the Hecke field $K(\varphi\rho(\hatK^{\times}))$, where $\hatK := K \otimes \varprojlim_{m} \ZZ/m\ZZ$ is the finite adelic ring and $\hatK^{\times}$ denotes its multiplicative group.

We associate the character $\chi$ with a theta series $\theta_{\varphi\rho} \in S_{2}(\Gamma_0(N_{\varphi\rho}))$ by the theory of Hecke and Shimura (see, for example \cite{MR453647} for an introduction). It satisfies
\[
L(\theta_{\varphi\rho}, s) = L(\varphi\rho, s).
\]
Let $M_{\rho}$ be the Hecke field of $\theta_{\varphi\rho}$. To the theta series $\theta_{\varphi\rho}$, we further associate it with an abelian variety $A_{\varphi\rho}$ over $\QQ$ by the construction of Eichler and Shimura, characterized up to isogeny by
\[
L(A_{\varphi\rho}/\QQ, s) = \prod_{\sigma: M_{\rho} \hookrightarrow \CC} L(\theta_{\varphi\rho}^{\sigma}, s).
\]
where $\theta_{\varphi\rho}^{\sigma}$ is the $\sigma$-conjugate of the theta series $\theta_{\varphi\rho}$. The following properties hold:
\begin{itemize}
    \item The Hecke field $M_{\rho}$ is a totally real number field, and $\dim A_{\varphi\rho} = [M_{\rho}:\QQ]$.
    \item The abelian variety $A_{\varphi\rho}$ is of $\GL_2$-type and $\End^{0}_{\QQ}(A_{\varphi\rho})$ is naturally identified with $M_{\rho}$. The field $\Phi_{\rho}$ is then an imaginary quadratic extension of $M_{\rho}$.
\end{itemize}
When $\rho = \bfone$ is the trivial character, we write $A_{\varphi}$, $M$ and $\Phi$ for $A_{\varphi\rho}$, $M_{\rho}$ and $\Phi_{\rho}$ respectively. The abelian variety $A_{\varphi}$ is the central object in this paper.

Let $B_n$ denote the Weil restriction $\Res_{K_n^{\ac}/K}(A_{\varphi})$ of the abelian variety $A_\varphi$ over $K_{n}^{\ac}$. As explained in \cite[Proof of Theorem 3.9]{MR4742720}, we have the following isogeny
\begin{equation} \label{eq:Anappears}
B_n \sim A_n\times B_{n-1}
\end{equation}
of abelian varieties over $K$, where $A_n$ is an abelian variety determined by the $L$-function up to isogeny. For later references, we record from \cite[Proof of Theorem 3.9]{MR4742720} that
\[
L(B_n/K,s) = \prod_{\rho \in \frY_{n}} \prod_{\sigma: \Phi \hookrightarrow \CC}L(\varphi^{\sigma}\rho,s)
\]
and
\begin{equation} \label{eq:LAnoverK}
L(A_n/K,s)=\prod_{\rho \in \frY_{n}^{\dagger}} \prod_{\sigma: \Phi \hookrightarrow \CC} L(\varphi^{\sigma}\rho,s).  
\end{equation}
By \eqref{eq:Anappears}, the set of $K$-rational points is then given by
\[ 
A_\varphi(K_n^{\ac})\otimes \mathbb{Q} \cong (A_n(K) \otimes \mathbb{Q}) \oplus (A_\varphi(K_{n-1}^{\ac}) \otimes \mathbb{Q}),
\]
and therefore
\begin{equation} \label{eq:rankdiffer}
    \rank_{\ZZ} A_{\varphi}(K_n^{\ac}) = \rank_{\ZZ} A_n(K) + \rank_{\ZZ} A_{\varphi}(K_{n-1}^{\ac}).
\end{equation}
By this observation, it suffices to study the structure of $A_n$.

\begin{proposition}\label{lemma:isogeny}
Notations being as above, there is an isogeny
\begin{equation} \label{eq:isogenystatement}
A_n \sim \prod_{\rho \in \frZ_{n}} A_{\varphi\rho}
\end{equation}
of abelian varieties over $K$, where $\frZ_{n}$ is a subset of the set of Galois characters of level $n$, to be specified in the proof. Moreover, we have
    \begin{equation} \label{eq:sum1}
        \sum_{\rho \in \frZ_{n}} [M_\rho:\mathbb{Q}] = (p^n-p^{n-1}) [M:\mathbb{Q}],    
    \end{equation}
and
    \begin{equation} \label{eq:sum2}
    \sum_{\rho \in \frZ_{n}} W(\varphi \rho) [M_{\rho} : \QQ] = \dfrac{1}{2} \sum_{\sigma: \Phi \hookrightarrow \CC} \sum_{\rho \in \frY_{n}^{\dagger}} W(\varphi^{\sigma} \rho).   
    \end{equation}
\end{proposition}

\begin{proof}
The proposition is proved by comparing $L$-functions of both sides. Recall
    \[
    L(A_{\varphi\rho}/\mathbb{Q}, s) = \prod_{\tau: M_{\rho} \hookrightarrow \CC} L(\theta_{\varphi\rho}^{\tau}, s),
    \]
hence we have
    \begin{align} \label{eq:LArhooverK}
    L(A_{\varphi\rho}/{K}, s) &= L(A_{\varphi\rho}/{\QQ}, s) L(A_{\varphi\rho}^{K}/{\QQ}, s) \nonumber \\ 
    &= \prod_{\tau: M_{\rho} \hookrightarrow \CC} L(\theta_{\varphi\rho}^{\tau}, s) L(\theta_{\varphi\rho}^{\tau \circ \sfc}, s) = \prod_{\tau: \Phi_{\rho} \hookrightarrow \CC} L((\varphi\rho)^{\tau}, s).
    \end{align}
Here $A_{\varphi\rho}^{K}$ is the quadratic twist of $A_{\varphi\rho}$ by $K$, and we have used that $L(\theta_{\varphi},s)=L(\varphi,s)$.

We now compare the right-hand-side of \eqref{eq:LAnoverK} and \eqref{eq:LArhooverK}:
\begin{itemize}
    \item The right-hand-side of \eqref{eq:LAnoverK} is a product of $(p^n - p^{n-1})[\Phi:\mathbb{Q}]$ many Hecke $L$-functions, with $\varphi$ runs over all its conjugates $\varphi^{\sigma}$ and $\rho$ runs over all Galois characters of level $n$.
    \item On the other hand, the right-hand-side of \eqref{eq:LArhooverK} is a product of $[\Phi_{\rho}:\mathbb{Q}]$ many Hecke $L$-functions, with $\varphi\rho$ runs over all conjugates of $\varphi\rho$ under $\tau: \Phi_{\rho} \hookrightarrow \CC$.
\end{itemize}
A trivial yet crucial observation is that every $(\varphi\rho)^{\tau}$ is of the form $\varphi^{\sigma} \rho^{\prime}$ for some $\sigma: \Phi \hookrightarrow \CC$ and $\rho^{\prime}$ a Galois character of level $n$. This implies that $L(A_{\varphi\rho}/K,s)$ is a factor of $L(A_n/K,s)$. They are not necessarily equal unless 
    \begin{equation} \label{eq:numberorbit}
    (p^n - p^{n-1}) [\Phi:\mathbb{Q}]=[\Phi_\rho:\mathbb{Q}].
    \end{equation}

If \eqref{eq:numberorbit} holds, then we have actually shown that $A_n$ is isogenous to $A_{\varphi\rho}$ over $K$ for a single $\rho$ of level $n$, which does not depend on the choice of $\rho$. Otherwise, we denote $\rho_1 := \rho$ and there is a factor
\[
L(\varphi \rho_2, s) \mid L(A_n/K, s), \quad \text{but } L(\varphi \rho_2, s) \nmid L(A_{\varphi  \rho_1}/K, s).
\]
Then $L(A_{\varphi \rho_2}/K, s)$ is also a factor of $L(A_n/K, s)$ which has no common factor with $L(A_{\varphi \rho_1}/K, s)$ since the characters appearing in their expressions are not conjugate under $\tau: \Phi_{\rho} \hookrightarrow \CC$. By repeating this process, extracting factors of the form $L(A_{\varphi \rho}/K, s)$ from $L(A_n/K, s)$, we obtain
    \begin{equation} \label{eq:factorizeAn}
    L(A_n/K,s)=\prod_{i=1}^t L(A_{\varphi\rho_i}/K,s),
    \end{equation}
where $\rho_i$ are the Galois characters of level $n$ chosen in this process. Let $\frZ_{n} = \{\rho_1, \ldots, \rho_t\}$. We read the isogeny of abelian varieties over $K$ from the factorization \eqref{eq:factorizeAn}, which is exactly the first claim of this proposition.

We prove \eqref{eq:sum1} by comparing the dimension of abelian varieties on both sides of the isogeny factorization of $A_n$ (i.e. \eqref{eq:isogenystatement}). Firstly,
\begin{align*}
    \dim A_{n} &= \dim B_{n} - \dim B_{n-1} \\
    &= [K_{n}^{\ac}:K] \dim A_{\varphi} - [K_{n-1}^{\ac}:K] \dim A_{\varphi} \\
    & = (p^n - p^{n-1}) [M:\QQ].
\end{align*}
Here the first equality follows from the definition of $A_n$ in \eqref{eq:Anappears} and the second equality follows from the definition of $B_n$ as the Weil restriction of $A_{\varphi}$ from $K_{n}^{\ac}$ to $K$. Secondly, $\dim A_{\varphi \rho} = [M_{\rho}: \QQ]$. It then follows from \eqref{eq:isogenystatement} that
    \[
    (p^n-p^{n-1}) [M:\mathbb{Q}]=\sum_{\rho\in \frZ_n} [M_\rho:\mathbb{Q}],
    \]
as desired.

To prove \eqref{eq:sum2}, we see that the construction of the set $\frZ_{n}$ gives us the decomposition
\[
\{\varphi^{\sigma} \rho: \sigma: \Phi \hookrightarrow \CC, \rho \in \frY_{n}^{\dagger} \} = \bigsqcup_{\rho \in \frZ_{n}} \{ (\varphi \rho)^{\tau}: \tau: \Phi_{\rho} \hookrightarrow \CC \}.
\]
Taking the summation of root numbers, we then have
\[
\sum_{\sigma: \Phi \hookrightarrow \CC} \sum_{\rho \in \frY_{n}^{\dagger}} W(\varphi^{\sigma} \rho) = \sum_{\rho \in \frZ_{n}} \sum_{\tau: \Phi_{\sigma} \hookrightarrow \CC} W((\varphi \rho)^{\tau}).
\]
Since root numbers are preserved under field automorphisms, it follows that
\[
\sum_{\rho \in \frZ_{n}} \sum_{\tau: \Phi_{\sigma} \hookrightarrow \CC} W((\varphi \rho)^{\tau}) = \sum_{\rho \in \frZ_{n}} [\Phi_{\rho} : \QQ] W(\varphi \rho).
\]
The equality \eqref{eq:sum2} then follows by noticing that $[\Phi_{\rho}: \QQ] = 2 [M_{\rho}: \QQ]$.
\end{proof}

In light of the isogeny decomposition in Proposition \ref{lemma:isogeny}, we apply the Gross--Zagier--Kolyvagin theory to the abelian varieties $A_{\varphi \rho_i}$ for $\rho_i \in \frZ_n$. In the case of elliptic curves $E/\QQ$, this theory originates in the work of Gross and Zagier \cite{gross1986heegner}, which relates the N\'eron--Tate height of Heegner points $y_K \in E(K)$ to the derivative $L'(E/K,1)$. Building on this, Kolyvagin developed his Euler system method (also known as the \emph{Kolyvagin machinery}) in \cite{kolyvagin1989finiteness, kolyvagin1990euler}, using Heegner points to bound the Shafarevich--Tate group and control the Mordell--Weil rank. These results were extended to modular abelian varieties over $\QQ$ in \cite{kolyvagin1990finiteness}, which is the form of the result used in this article. A generalization to totally real number fields is given in \cite[Theorem A]{MR1826411}. Although that work treats a broader setting, we only use the classical case over $\QQ$.

\begin{theorem}[Gross-Zagier, Kolyvagin, Kolyvagin-Logachev, Zhang]\label{lemma:GZK}
    Suppose that 
    \[
    r_{\varphi\rho} := \ord_{s=1}L(\varphi\rho,s) \in \{0,1\},
    \]
    then $\rank_{\mathbb{Z}}A_{\varphi\rho}(\mathbb{Q})=r_{\varphi\rho} [M_\rho:\mathbb{Q}]$.
\end{theorem}

\begin{proof}[Proof of Theorem \ref{thm:mainA}]
    By \eqref{eq:rankdiffer}, it suffices to compute $\rank A_n(K)$. It follows from Proposition \ref{lemma:isogeny} that
    \[
    \rank_{\ZZ} A_n(K) = \sum_{\rho \in \frZ_{n}} \rank_{\ZZ} A_{\varphi\rho}(K) = 2 \sum_{\rho \in \frZ_{n}} \rank_{\ZZ} A_{\varphi\rho}(\mathbb{Q}).
    \]
Here for the second equality, we observe that the abelian variety $A_{\varphi\rho}$ is isogenous to its twist by the quadratic character of $K/\QQ$. By Theorem \ref{thm:RJ}, there exists a positive integer $N_{\JR}$ such that for all $n > N_{\JR}$, we have \footnote{The subscript “RJ” refers to D. Rohrlich \cite{MR735332} and H. Jia \cite{jia2024lfunctionsheckecharactersanticyclotomic} who proved Theorem \ref{thm:RJ}.}  
    \[
    r_{\varphi\rho} = \tilW(\varphi \rho) := \dfrac{1-W(\varphi\rho)}{2} \in \{0, 1\},
    \]
and hence Theorem \ref{lemma:GZK} is applicable for $n > N_{\JR}$, which implies
\[
    \rank A_n(K) = 2 \sum_{\rho \in \frZ_{n}} [M_{\rho}:\QQ] r_{\varphi\rho} = \sum_{\rho \in \frZ_{n}} [M_{\rho}:\QQ] (1-W(\varphi \rho))
    \]
for $n > N_{\JR}$.
We discuss the following three cases separately.

    \textbf{Case 1: $p$ splits in $K$}. By Theorem \ref{thm:rootnumber}(i), $W(\varphi \rho) = W(\varphi)$, which is independent of the choice of $\rho$. Thus, for $n > N_{\JR}$, 
    \[
    \rank A_n(K)= \sum_{\rho\in \frZ_n} [M_\rho:\mathbb{Q}] (1-W(\varphi)) = (1-W(\varphi)) [M:\mathbb{Q}] \phi(p^n).
    \]
    Therefore,
    \[
    \rank_{\ZZ} A_\varphi(K_n^{\ac}) = 
    \begin{cases}
    \rank_{\ZZ} A_\varphi(K_{n-1}^{\ac}), &\text{if } W(\varphi) = 1, \\
    2 [M:\QQ] \phi(p^n) + \rank_{\ZZ} A_\varphi(K_{n-1}^{\ac}), &\text{if } W(\varphi) = -1.
    \end{cases}
    \]

    \textbf{Case 2: $p$ remains inert in $K$}.  
    By Theorem \ref{thm:rootnumber}(ii), $W(\varphi \rho) = (-1)^{n-j+1-f(\varphi_p)} W(\varphi)$, which is also independent of the choice of $\rho \in \frZ_{n}$ for $n > N_{\inert}$ in Lemma \ref{lemma:pinert}. Therefore, as argued in \textbf{Case 1}, for $n > \max \{ N_{\JR}, N_{\inert, \varphi} \}$, we have the following:
    \begin{enumerate}[(a)]
        \item When $j + f(\varphi_p)$ is even,
    \[
    \rank_{\ZZ} A_\varphi(K_n^{\ac}) = 
    \begin{cases}
    2 [M:\QQ] \phi(p^n) + \rank_{\ZZ} A_\varphi(K_{n-1}^{\ac}), &\text{if } n \equiv \tilW(\varphi) \bmod 2, \\
    \rank_{\ZZ} A_\varphi(K_{n-1}^{\ac}), &\text{if } n \not\equiv \tilW(\varphi) \bmod 2.
    \end{cases}
    \]
        \item When $j + f(\varphi_p)$ is odd,
        \[
    \rank_{\ZZ} A_\varphi(K_n^{\ac}) = 
    \begin{cases}
    2 [M:\QQ] \phi(p^n) + \rank_{\ZZ} A_\varphi(K_{n-1}^{\ac}), & n \not\equiv \tilW(\varphi) \bmod 2, \\
    \rank_{\ZZ} A_\varphi(K_{n-1}^{\ac}), & n \equiv \tilW(\varphi) \bmod 2.
    \end{cases}
    \]
    \end{enumerate}

    \textbf{Case 3: $p$ ramifies in $K$}.  
    We consider 
    \begin{align*}
    \sum_{\rho \in \frZ_{n}} (1-W(\varphi\rho))[M_\rho:\mathbb{Q}] &= \sum_{\rho \in \frZ_{n}} [M_\rho:\mathbb{Q}] - \sum_{\rho \in \frZ_{n}} W(\varphi\rho)[M_\rho:\mathbb{Q}] \\
    &= [M:\QQ] \phi(p^n) - \dfrac{1}{2} \sum_{\sigma: \Phi \hookrightarrow \CC} \sum_{\rho \in \frY_{n}^{\dagger}} W(\varphi^\sigma \rho).
    \end{align*}
Here the second equality follows from equations \eqref{eq:sum1} and \eqref{eq:sum2} in Proposition \ref{lemma:isogeny}. Furthermore, for every fixed character $\varphi^{\sigma}$, we use Theorem \ref{thm:rootnumber}(iv)(v) to see that for all $\sigma: \Phi \hookrightarrow \CC$,
    \[
    \sum_{\rho \in \frY_{n}^{\dagger}} W(\varphi^\sigma \rho)=0,
    \]
for $n > \max\{N_{\mathrm{ram}, \varphi^{\sigma}}: \sigma: \Phi \hookrightarrow \CC\}$ defined in Lemma \ref{lemma:pramified}, since half of the invertible residue classes modulo $p$ are quadratic residues, and the other half are non-residues, and hence half of the root numbers $W(\varphi^{\sigma} \rho)$ are $1$ and the other half of them are $-1$. Therefore,
    \[
    \rank_{\ZZ} A_\varphi(K_n^{\ac}) = [M:\QQ] \phi(p^n) + \rank_{\ZZ} A_\varphi(K_{n-1}^{\ac})
    \]
for $n > \max\{ N_{\JR}, N_{\mathrm{ram}, \varphi^{\sigma}}: \sigma: \Phi \hookrightarrow \CC  \}$. This completes the proof of Theorem \ref{thm:mainA}. 

In particular, in all three cases, we observe the congruence relation
    \[
    \rank_{\ZZ} A_{\varphi}(K_n^{\ac}) \equiv \rank_{\ZZ} A_{\varphi}(K_{n-1}^{\ac}) \mod{\phi(p^n)},
    \]
as stated in the introduction.
\end{proof}

\begin{remark} \label{rem:matches} 
Theorem \ref{thm:mainA} aligns with the results for elliptic curves recalled in Section \ref{sec:previous}. Specifically, in the case of elliptic curves $E$, the Hecke field $M$ is $\QQ$, implying $[M:\QQ] = 1$. By the theory of complex multiplication, $K$ must have class number one, so $j = 0$. Furthermore, when $E$ has good reduction at $p$, the associated Hecke character $\varphi$ of $E$ is unramified at $p$. Therefore $j+f(\varphi_p) = 0$, an even number. These conditions ensure that our results are consistent with the previously known results in this case (for example, \cite[Theorem 1.7, Theorem 1.8]{MR1860044}).
\end{remark}

\section{On the finiteness of the Mordell-Weil group} \label{sec:finitetor}
So far in this article, we have focused on the growth pattern of the rank of $A$ along the anticyclotomic tower $K_{n}^{\ac}/K$, ignoring the torsion points. In this section, we provide a proof that the group of torsion points $A(K_{\infty}^{\ac})_{\tors}$ is finite for \emph{any} abelian variety $A/\QQ$. This was observed in \cite[Exercise 1.10]{MR1860044} in the case of elliptic curves. Further motivation for this result can be found in Section \ref{sec:introfurther}.

\begin{lemma} \label{lem:inert}
Any prime of $K$ that is inert in $K/\QQ$ and does not lie above $p$ must split completely in $K_{\infty}^{\ac}/K$.
\end{lemma}

\begin{proof}
The following proof is adapted from \cite[Section 2]{MR357371}. Let $q$ be a rational prime coprime to $p$ that is inert in $K$, and let $\mathfrak{q}$ be the unique prime of $K$ lying above $q$. Then $\mathfrak{q}$ is unramified in $K_{n}^{\mathrm{ac}}$. Let $\mathfrak{q}_n$ be a prime of $K_{n}^{\mathrm{ac}}$ lying above $\mathfrak{q}$, and let $Z_n$ be the decomposition group of $\mathfrak{q}_n$ in $K_{n}^{\mathrm{ac}}/\mathbb{Q}$.

Since $q$ is unramified in $K_{n}^{\mathrm{ac}}$ and inert in $K$, the decomposition group $Z_n$ is cyclic and its order is divisible by $2$. On the other hand, $G_n := \mathrm{Gal}(K_{n}^{\mathrm{ac}}/\mathbb{Q})$ is a dihedral group of order $2p^n$. Hence any cyclic subgroup of $G_n$ whose order is divisible by $2$ must have order exactly $2$. Therefore $|Z_n|=2$ and $Z_n \cap H_n = 1$, where $H_n := \mathrm{Gal}(K_{n}^{\mathrm{ac}}/K)$. Hence $\mathfrak{q}$ splits completely in $K_{n}^{\mathrm{ac}}$.
\end{proof}

\begin{proposition}
Let $A/\QQ$ be an abelian variety over $\QQ$, then $A(K_{\infty}^{\ac})_{\tors}$ is finite.
\end{proposition}
\begin{proof}
We consider the set $S$ of rational primes $q$ such that
\begin{enumerate}[(i)]
    \item The abelian variety $A/\QQ$ has good reduction at $q$,
    \item $q \neq p$,
    \item $q$ remains inert in $K$. Thus, by Lemma \ref{lem:inert}, $q$ splits completely in $K_{\infty}^{\ac}/K$.
\end{enumerate}
For any fixed place $q_{\infty}$ of $K_{\infty}^{\ac}$, its restriction $q_n$ on $K_{n}^{\ac}$ has residue field isomorphic to $\FF_{q^2}$. By the density theorem, we see that the set $S$ has density $1/2$, and hence $S$ is infinite.

Let $q \in S$. We first prove that the number of torsion points in $A(K_{\infty}^{\ac})$ of order prime to $q$ is finite. Let $P \in A(K_{\infty}^{\ac})[m]$ for some positive integer $m$ coprime to $q$. Then $P \in A(K_{n_P}^{\ac})[m]$ for a certain layer $n_P$. By \cite[page 495]{MR236190}, there is an injection
\[
A(K_{n_P}^{\ac})[m] \hookrightarrow \tilA(k_{q_{n_P}}) = \tilA(\FF_{q^2}),
\]
where $\tilA$ denotes the reduction of $A$ modulo $q$ \footnote{This is valid since $A$ has good reduction at $q$.}. We observe that the size of $\tilA(\FF_{q^2})$ is fixed and independent of $n_P$. Hence, $\cup_{q \nmid m} A(K_{\infty}^{\ac})[m]$ is finite. Indeed, suppose otherwise: for any positive integer $N$, there exist distinct points $P_1, \ldots, P_N \in A(K_{\infty}^{\ac})$ of order $m_1, \ldots, m_N$ prime to $q$. Then $P_i \in A(K_{\infty}^{\ac})[m]$ for $m = m_1 \cdots m_N$. Thus,
\[
\{P_1, \ldots, P_N\} \hookrightarrow \tilA(\FF_{q^2}).
\]
But this is impossible if $N > \# \tilA(\FF_{q^2})$.

To show that $A(K_{\infty}^{\ac})[q^{\infty}]$ is finite, since there are infinitely many primes in $S$, we can choose another prime $q^{\prime}$ and apply the previous argument. This shows that the number of torsion points in $A(K_{\infty}^{\ac})$ of order prime to $q^{\prime}$ is finite. In particular, $A(K_{\infty}^{\ac})[q^{\infty}]$ is finite.
\end{proof}

\begin{remark}
There do exist certain $\ZZ_p$-extensions $F_{\infty}/F$ such that $A(F_{\infty})_{\tors}$ is infinite for some abelian variety $A$. See \cite[Exercise 1.15]{MR1860044} for an example.
\end{remark}

Combining this with Theorem \ref{thm:mainA}, we obtain the following corollary.

\begin{corollary} \label{coro:finitelygenerated}
Let $A_{\varphi}$ be the abelian variety defined above. Then $A_{\varphi}(K_{\infty}^{\ac})$ is a finitely generated abelian group if and only if $p$ splits in $K$ and $W(\varphi) = 1$.
\end{corollary}

\section{Distribution of vanishing orders among anticyclotomic twists} \label{sec:distribution}

In this section, we study the distribution of the parity of vanishing orders of \( L(\varphi \rho, s) \) at \( s=1 \) as \( \rho \) varies among anticyclotomic Galois characters \( \rho: \Gal(K_{\infty}^{\ac}/K) \rightarrow \CC^{\times} \). An informal introduction to this topic was provided in Section \ref{sec:introdistribution}. The problem of determining this distribution pattern reduces to our computation of root numbers in Theorem \ref{thm:rootnumber}, thanks to Theorem \ref{thm:RJ}, in a manner analogous to the proof of Theorem \ref{thm:mainA}. 

To formalize this discussion, we define subsets of \( \frY_{n} \) as follows:
\begin{equation} \label{eq:defYnpm}
\frY_{n}^{\pm} := \{\rho \in \frY_{n} : \ord_{s=1} L(\varphi \rho, s) \equiv (1-\pm 1)/2 \pmod{2} \}.
\end{equation}
The ``frequency” of even and odd vanishing orders is represented by the sequences
\[
\left\{\bfP^{\pm}_{N} := \frac{\# \frY_{N}^{\pm}}{\# \frY_{N}} \right\}_{N \geq 0}.
\]

Using Theorem \ref{thm:rootnumber} and Theorem \ref{thm:RJ}, we establish the following results.

\begin{theorem}
The behavior of \( \bfP_{N}^{\pm} \) as \( N \to \infty \) can be described as follows.
\begin{enumerate}[\rm (1)]
    \item When \( p \) splits in \( K \), 
    \[
    \bfP_{N}^{+} \to \frac{W(\varphi)+1}{2}, \quad \bfP_{N}^{-} \to \frac{-(W(\varphi)-1)}{2}.
    \]
    \item When \( p \) remains inert in \( K \), we show the limit of each $\bfP_{?}^{?}$ as $k \rightarrow \infty$ in the following table.
    
\begin{table}[H]
\centering
\begin{tabular}{|cc|c|c|c|c|}
\hline
\multicolumn{2}{|c|}{}                                                               & $\bfP_{2k+1}^{+}$ & $\bfP_{2k}^{+}$ & $\bfP_{2k+1}^{-}$ & $\bfP_{2k}^{-}$ \\ \hline
\multicolumn{1}{|c|}{\multirow{2}{*}{$j+f(\varphi_p)$ is even}}  & $W(\varphi) = 1$  & $\dfrac{1}{p+1}$   & $\dfrac{p}{p+1}$ & $\dfrac{p}{p+1}$   & $\dfrac{1}{p+1}$ \\ \cline{2-6} 
\multicolumn{1}{|c|}{}                                           & $W(\varphi) = -1$ & $\dfrac{p}{p+1}$   & $\dfrac{1}{p+1}$ & $\dfrac{1}{p+1}$   & $\dfrac{p}{p+1}$ \\ \hline
\multicolumn{1}{|c|}{\multirow{2}{*}{$j+f(\varphi_p)$ is odd}} & $W(\varphi) = 1$  & $\dfrac{p}{p+1}$   & $\dfrac{1}{p+1}$ & $\dfrac{1}{p+1}$   & $\dfrac{p}{p+1}$ \\ \cline{2-6} 
\multicolumn{1}{|c|}{}                                           & $W(\varphi) = -1$ & $\dfrac{1}{p+1}$   & $\dfrac{p}{p+1}$ & $\dfrac{p}{p+1}$   & $\dfrac{1}{p+1}$ \\ \hline
\end{tabular}
\end{table}
In particular,
\[
\liminf_{N \rightarrow \infty} \bfP_{N}^{\pm} = \dfrac{1}{p+1}, \quad \limsup_{N \rightarrow \infty} \bfP_{N}^{\pm} = \dfrac{p}{p+1},
\]
and both \( \{\bfP_{N}^{+}\} \) and \( \{\bfP_{N}^{-}\} \) diverge as \( N \to \infty \).
    \item When \( p \) ramifies in \( K \),
    \[
    \bfP_{N}^{+}, \, \bfP_{N}^{-} \to \frac{1}{2}.
    \]
\end{enumerate}
\end{theorem}

\begin{proof}
By Theorem \ref{thm:RJ}, the order of $L(\varphi \rho, s)$ is either zero or one, which is completely determined by the root number of $W(\varphi \rho)$, when the level $n$ of $\rho$ is sufficiently large. Moreover, in all the cases of decomposition type of $p$ in $K$, the pattern of the quotient $R(\varphi, \rho)$ stabilizes for sufficiently large $n$. Combining these two observations, the result of the theorem follows from Theorem \ref{thm:rootnumber} by a straightforward computation. We only make two remarks on this.
\begin{itemize}
    \item When $p$ remains inert in $K$, the behavior of sequences $\bfP_{N}^{\pm}$ reduces to the following computation:
    \[
    \left(\sum_{0 \leq n \leq N, \, n \equiv 0 \bmod 2} \phi(p^n)\right)/{p^N} = \begin{cases}
    \dfrac{p+p^{-N}}{p+1}, &\quad \text{if } N \text{ is even}, \\
    \dfrac{1+p^{-N}}{p+1}, &\quad \text{if } N \text{ is odd}, \\
    \end{cases}
    \]
    and
    \[
    \left(\sum_{0 \leq n \leq N, \, n \equiv 1 \bmod 2} \phi(p^n)\right)/{p^N} = \begin{cases}
    \dfrac{1-p^{-N}}{p+1}, &\quad \text{if } N \text{ is even}, \\
    \dfrac{p-p^{-N}}{p+1}, &\quad \text{if } N \text{ is odd}. \\
    \end{cases}
    \]
    \item When $p$ ramifies in $K$, we use the observation that half of the invertible residue classes modulo $p$ are quadratic residues, and the other half are non-residues. This is the same observation used in the proof of Theorem \ref{thm:mainA}.
\end{itemize}
The case when $p$ splits in $K$ is straightforward.
\end{proof}

\small 
\printbibliography

\end{document}